\documentclass[a4paper, 12pt]{article}

\usepackage{amsmath} %odgovoran za \newcommand{\Bi}{\operatorname{Im}}
\usepackage{theorem}
\usepackage{amssymb}
\usepackage{amsfonts}
\usepackage{latexsym}
\usepackage{amscd}
\usepackage{graphicx}
\DeclareMathAlphabet{\mathpzc}{OT1}{pzc}{m}{it}

\newcommand{\cab}{\underline{{\mathcal C}}}
\newcommand{\otb}{{\overline{\otimes}}}
\newcommand{\otk}{{\otimes_{\ku}}}

\newcommand{\ku}{{\Bbbk}}
\renewcommand{\_}[1]{\mbox{$_{\left( #1 \right)}$}}
\newcommand{\op}{\rm{op}}
\newcommand{\rev}{\rm{rev}}
\newcommand{\cop}{\rm{cop}}
\newcommand{\no}{{\mathcal N}}
\newcommand{\mo}{{\mathcal M}}
\newcommand{\Kc}{{\mathcal K}}
\newcommand{\nic}{{\mathfrak B}}
\newcommand\Rep{\operatorname{Rep}}
\newcommand{\ele}{{\mathcal L}}

\newcommand{\gr}{\mbox{\rm gr\,}}
\newcommand{\brp}{\mbox{\rm BrPic\,}}
\newcommand{\biga}{\mbox{\rm BiGal\,}}
\newcommand{\diag}{\,\text{diag}}
\newcommand{\Ss}{{\mathcal S}}

\newcommand{\Fc}{{\mathcal F}}

\newcommand{\co}{\rm{co}}
\newcommand{\ca}{{\mathcal C}}
\newcommand{\Do}{{\mathcal D}}
\newcommand\Comod{\operatorname{Comod}}
\newcommand{\uno}{{\mathbf 1}}

\usepackage{xspace,colortbl}
\usepackage{color}

\definecolor{verde}{rgb}{0.,0.7,0.}
\definecolor{indigo}{rgb}{.18, .34, .78}
\definecolor{indigo1}{rgb}{.18, .24, .78}
\definecolor{indigo2}{rgb}{.18, .14, .78}
\definecolor{indigo3}{rgb}{.18, 0., .78}
\definecolor{rojo}{rgb}{1,0,0}
\definecolor{negro}{rgb}{0,0,0}
\definecolor{lila}{rgb}{.46, .16, .78}
\definecolor{lila1}{rgb}{.46, .16, .86}
\definecolor{lila2}{rgb}{.56, .16, .86}
	\definecolor{lila3}{rgb}{.63, .16, .78}
\definecolor{lila4}{rgb}{.7, .16, .78}
\definecolor{lila5}{rgb}{.78, .26, .78}
\definecolor{lila6}{rgb}{.6, 0., .78}

\theoremstyle{plain}
\theoremheaderfont{\normalfont\bfseries}

\newtheorem{thm}{Theorem}[section]
\newtheorem{claim}[thm]{Claim}
\newtheorem{lma}[thm]{Lemma}
\newtheorem{cor}[thm]{Corollary}

\newtheorem{defn}[thm]{Definition}

\theorembodyfont{\rmfamily}

\newtheorem{rem}[thm]{Remark}
\newtheorem{prop}[thm]{Proposition}
\newcommand{\qed}{\hfill\quad\fbox{\rule[0mm]{0,0cm}{0,0mm}}  \par\bigskip}

\newcommand{\x}{\mbox{-}}

\newcommand{\iso}{\cong}
\newcommand{\ot}{\otimes}
\newcommand{\ots}{\small{\otimes}}
\newcommand{\C}{{\mathcal C}}
\newcommand{\M}{{\mathcal M}}
\newcommand{\D}{{\mathcal D}}
\newcommand{\F}{{\mathcal F}}
\newcommand{\G}{{\mathcal G}}

\newcommand{\HH}{{\mathcal H}}

\newcommand{\B}{{\mathcal B}}

\def\ul{\underline}

\newcommand{\Ll}{{\mathcal L}}
\newcommand{\Pp}{{\mathcal P}}
\newcommand{\ch}{{\mathpzc h}}

\newcommand{\crta}{\overline}

\newcommand{\id}{\operatorname {id}}

\newcommand{\Hom}{\operatorname {Hom}}
\newcommand{\End}{\operatorname {End}}
\newcommand{\Fun}{\operatorname {Fun}}
\newcommand{\Aut}{\operatorname {Aut}}
\def\N{{\mathbb N}}

\def\Zz{{\mathbb Z}}  %"celi brojevi"
\def\Z{{\mathcal Z}}  %"kategorija"
\newcommand{\bt}{\boxtimes}

\newcommand{\inbi}{\operatorname{InnbiGal}}
\newcommand{\Mod}{\operatorname{Mod}}

\newcommand{\cref}[1]{C.~\ref{c:#1}}

\newcommand{\lelabel}[1]{\label{le:#1}}
\newcommand{\leref}[1]{Lemma~\ref{le:#1}}
\newcommand{\eqlabel}[1]{\label{eq:#1}}
\newcommand{\equref}[1]{(\ref{eq:#1})}
\newcommand{\thlabel}[1]{\label{th:#1}}
\newcommand{\thref}[1]{Theorem~\ref{th:#1}}
\newcommand{\delabel}[1]{\label{de:#1}}
\newcommand{\deref}[1]{Definition~\ref{de:#1}}
\newcommand{\prlabel}[1]{\label{pr:#1}}
\newcommand{\prref}[1]{Proposition~\ref{pr:#1}}
\newcommand{\colabel}[1]{\label{co:#1}}
\newcommand{\coref}[1]{Corollary~\ref{co:#1}}

\newcommand{\rmlabel}[1]{\label{rm:#1}}
\newcommand{\rmref}[1]{Remark~\ref{rm:#1}}
\newcommand{\selabel}[1]{\label{se:#1}}

\begin{document}

\title{Invertible bimodule categories over the representation category of a Hopf  algebra}
\author{Bojana Femi\'c $^{a}$, Adriana Mej\'ia Casta\~no $^{b}$ and
Mart\'\i n Mombelli  $^{b}$
 \vspace{6pt} \\
$^{a}${\small Facultad de Ingenier\'ia, \vspace{-2pt}}\\
{\small  Universidad de la Rep\'ublica} \vspace{-2pt}\\
{\small  Julio Herrera y Reissig 565, 11 300 Montevideo, Uruguay}\\
 \vspace{6pt} \\
$^{b}${\small  CIEM-FAMAF, }\\
 {\small  Universidad Nacional de C\'ordoba}\\
 {\small  Medina Allende s/n, 5000, C\'ordoba, Argentina}}

\maketitle
\begin{abstract} For any finite-dimensional Hopf algebra $H$ we construct a 
group homomorphism $\biga(H)\to \text{BrPic}(\Rep(H))$, from the group
of equivalence classes of $H$-biGalois objects to the group of equivalence classes
of invertible exact $\Rep(H)$-bimodule categories. We discuss the injectivity of this map. 
We exemplify in the case $H=T_q$ is a Taft Hopf algebra and for this we classify all exact
indecomposable $\Rep(T_q)$-bimodule categories.
 \bigbreak
\bigbreak
\bigbreak
\bigbreak
{\em Mathematics Subject Classification (2010): 18D10, 16W30, 19D23.}

{\em Keywords: Brauer-Picard group, tensor category, biGalois object.}
\end{abstract}

\section{Introduction}

The Brauer-Picard group $\brp(\C)$ of a finite tensor category $\C$ is the group
 of equivalence classes of invertible exact $\C$-bimodule categories. 
This group, and its higher versions, were introduced in [10] to classify 
extensions of a given tensor category $\C$ by a finite group. Also, it has a 
close relation to certain structures appearing in mathematical physics, 
like rational Conformal Field Theory or 3-dimensional Topological Field Theory,
 see for example [12], [9], [16]. 

In order to classify extensions of a tensor category 
$\C$ by a finite group $G$ one needs a group map $G\to\brp(\C)$ and  certain 
cohomological data, \cite{ENO}. Henceforth, determining any subgroup 
of $\brp(\C)$ presents a significant step in the mentioned classification.

The Brauer-Picard group is a complicated group to compute even in the simplest examples. 
One interesting problem is the computation of $\brp(\Rep(H))$, where $H$ is a finite-dimensional 
Hopf algebra. It is known that any exact $\Rep(H)$-bimodule category is equivalent to the category
 of finite-dimensional representations of a left $H\ot H^{\cop}$-comodule algebra. Given a 
left $H\ot H^{\cop}$-comodule algebra the problem of deciding when its category of representations 
is invertible is not solved. The main difficulty is that the Deligne's tensor product of bimodule 
categories is not easy to compute explicitly.
\medbreak

One of the principle goals of this paper is the description of a certain family of invertible 
exact $\Rep(H)$-bimodule categories coming from $H$-biGalois objects. They give rise to a subgroup
 of $\brp(\Rep(H))$. This result is expressed in \coref{bigal seq}. 
Here we construct a short exact sequence involving a map from the group of
 isomorphism classes of $H$-biGalois objects $\biga(H)$ to $\brp(\Rep(H))$ and its kernel. As a 
consequence we get that for any co-quasitriangular Hopf algebra $H$ the group 
$\biga(H)$ embedds into the Brauer-Picard  group $\brp(\Rep(H))$.\medbreak

The subsequent part of the paper (the Section 5) is dedicated to the study of the
 case when $H$ is the Taft Hopf algebra 
$T_q=\ku\langle g, x\vert g^n=1, x^n=0, gx=qxg\rangle$, where $q$ is a primitive $n$-th 
 root of unity. Our second main goal is to classify all exact indecomposable
 $\Rep(T_q)$-bimodule categories; we obtain five families of them. 
As announced, some families of invertible bimodule
 categories  arise from biGalois objects. %As $T_q$ is a finite-dimensional Hopf algebra,
% all its biGalois objects are cleft, that is, they are isomorphic to $T_q$ as $T_q$-bicomodules.
 %Then it becomes clear that from the five families of exact indecomposable 
%$\Rep(T_q)$-bimodule categories that we obtained a subclass of one of them emerges 
%from biGalois objects.

The approach we use to classify the exact indecomposable $\Rep(T_q)$-bimodule 
categories is the following. Let $H$ be a finite-dimensional Hopf algebra. Any exact 
indecomposable $\Rep(H)$-bimodule category is equivalent to the category of 
finite-dimensional representations of a left $H\ot H^{cop}$-comodule algebra which 
is $H\ot H^{cop}$-simple (it has no non-trivial ideals which are simultaneously left 
$H\ot H^{cop}$-comodules) and with trivial coinvariants. So it is enough to find all 
comodule algebras over $H\ot H^{cop}$ with these properties. 
By \cite{Sk} any coideal subalgebra $A$ of $H$ is $H$-simple,
 and due to \cite[Remark 3.2]{M1}
the representation category of $A$ twisted by a {\em compatible} 2-cocycle over 
$H$ is an exact indecomposable $\Rep(H)$-bimodule category. Morevover, the \emph{liftings} 
of cocycle twisted coideal subalgebras are as well $H$-simple comodule algebras 
with trivial coinvariants. Here, a \emph{lifting} of a coideal subalgebra $K$ is a comodule algebra
 $A$ such that $\gr A\simeq K$. Then setting $H=T_q\ot T_q^{cop}$, we determine all
 homogeneous coideal subalgebras of $H$ and their 2-cocycle twists. 
We determine the liftings of them and get five families of $H$-comodule algebras.
 By the above, the representation categories of these five families are exact 
indecomposable $\Rep(T_q)$-bimodule categories. In \thref{ex-modcat} 
we prove that every exact indecomposable $\Rep(T_q)$-bimodule category is of this form.
 This classification result is interesting in itself. The biGalois objects over $T_q$ 
arise from one of the five families of the above comodule algebras and from their
 form it is straightforward that the kernel of the map $\biga(T_q)\to\brp(\Rep(T_q))$
 is trivial. Hence, the  group $\biga(T_q)$ embedds into $\brp(\Rep(T_q))$, 
although $T_q$ is not co-quasitriangular.\medbreak

The contents of the paper are the following. In Section 2 we give 
the necessary preliminaries on tensor categories and their representations. 
We also prove that for Galois objects over Hopf algebras, the tensor product
 of bimodule categories can be given in an explicit form. In Section 3 we recall 
some basic notions on bicategories that we use later. We also give a proof of a
known result; that the bicategory of representations of a given tensor category
 determines the tensor category up to Morita equivalence. In Section 4 we present a 
group homomorphism $\biga(H) \to \brp(\Rep(H))$. The major part of this section is dedicated 
to describe its kernel as well as possible. For co-quasitriangular Hopf algebras this
 map is always injective. In Subsection 5.1 we compute the 2-cocycle twists of 
$H= T_q\ot T_q^{cop}$. In the next subsection we find all homogeneous coideal 
subalgebras of $H$. In Subsection 5.3 we introduce five families of $H$-comodule
 algebras that are $H$-simple and with trivial coinvariants which turn out to be
 liftings of the coideal subalgebras. In Subsection 5.4 we classify all exact 
indecomposable $\Rep(T_q)$-bimodule categories and prove that they all 
come from the above five families of comodule algebras. We also determine 
the biGalois objects over $T_q$. The last subsection is dedicated to the explicit 
embedding of $\ku^{\times} \ltimes\ku^{+}\simeq\biga(T_q)$ into $\brp(\Rep(T_q))$.

\section{Preliminaries and Notation}\selabel{prelimi}

We shall work over an algebraically closed field $\ku$ of characteristic zero. If $G$ is a finite
group and $\psi\in Z^2(G,\ku^{\times})$ is a 2-cocycle there is another 2-cocycle $\psi'$
in the same cohomology class as $\psi$ such that
\begin{equation}\eqlabel{2-cocyclo-g}
 \psi'(g,1)=\psi'(1,g)=1, \quad \psi'(g,g^{-1})=1, \quad \psi'(g,h)^{-1}=\psi'(h^{-1}, g^{-1}),
\end{equation}
for all $g,h\in G$. 
\medbreak

All vector spaces and algebras are assumed to be over $\ku$.  We denote by $vect_\ku$ the category of
finite-dimensional $\ku$-vector spaces. If $A$ is an algebra
we shall denote by ${}_A\mo$ the category of finite-dimensional left
$A$-modules.

If $C$ is a coalgebra and $V$ is a right $C$-comodule with comodule structure
$\rho_V:V\to V\otk C$
and $W$ is a left $C$-comodule with comodule structure $\lambda_W:W\to C\otk W$, we denote 
 by 
$W^{\co C}=\{w\in W: \lambda_W(w)=1\ot w\}$ the set of (left) coinvariants of $W$  
and $V\Box_C W$ is the equalizer 
of the arrows $\rho_V\ot\id, \id\ot  \lambda_W: V\otk W\to  V\otk C\otk W$.

Let $H$ be a finite-dimensional Hopf algebra. We denote by
$G(H)$ the group of group-like elements in $H$. We shall denote by $\Rep(H)$ the tensor category
of finite dimensional left $H$-modules and $\Comod(H)$ the tensor category
of finite dimensional left $H$-comodules. For the basic notions and definitions of 
Hopf Galois extensions over a Hopf algebra the reader is referred to \cite{S2}.

\subsection{ Hopf algebras and comodule algebras}

Given a coradically graded Hopf algebra $H=\oplus_{i=0}^m H(i)$  we say that
a left coideal subalgebra $K\subseteq H$ is \emph{homogeneous}
if it is a graded algebra  $K=\oplus_{i=0}^m K(i)$ such that
$K(i)\subseteq H(i)$.

Let  $H$ be a finite-dimensional Hopf algebra. If $(K, \lambda)$ is a left $H$-comodule algebra
we denote by $\overline{K}$ the right $H$-comodule algebra with opposite underlying
algebra $K^{\op}$ and coaction $\overline{\lambda}: K\to   K\otk H$ given by
$$\overline{\lambda}(k)=k\_{0}\ot \Ss^{-1}(k\_1), \quad \text{ for all } k\in K.$$

If $L, K$ are right $H$-comodule algebras, we denote by ${}_L\mo^H_{K}$ the
category of $(L,K)$-bimodules with a right $H$-comodule structure such that 
it is a morphism of $(L,K)$-bimodules. If
$L, K$ are left $H$-comodule algebras the category ${}_L^H\mo_{K}$ can be defined
similarly. 

 An $H$-comodule algebra is said to be {\em $H$-simple } 
if it has no non-trivial $H$-costable ideals. 

\subsection{Twisting Hopf algebras}

Let $H$ be a Hopf algebra. Let us recall that a Hopf 2-cocycle for $H$ is a  map $\sigma: H\otk H\to
\ku$, invertible with respect to convolution,  such that
\begin{align}\label{2-cocycle}
\sigma(x\_1, y\_1)\sigma(x\_2y\_2, z) &= \sigma(y\_1,
z\_1)\sigma(x, y\_2z\_2),
\\
\label{2-cocycle-unitario} \sigma(x, 1) &= \varepsilon(x) =
\sigma(1, x),
\end{align}
for all $x,y, z\in H$. Using this cocycle there is a new Hopf
algebra structure constructed over the same coalgebra $H$ with the
product described by
\begin{equation} \eqlabel{sigma-mult}
x\cdot_{\sigma}y=\sigma(x_{(1)}, y_{(1)})\sigma^{-1}(x_{(3)}, y_{(3)})x_{(2)} y_{(2)}, \quad x,y\in H
\end{equation}
This new Hopf algebra is denoted by  $H^{[\sigma]}$.  If $\sigma: H\otimes H\to \ku$ is a Hopf 2-cocycle and $A$ is a left
$H$-comodule algebra, then there is a new product in $A$ given by
\begin{align}\label{sigma-product} a\cdot_{\sigma}b = \sigma(a_{(-1)},
b_{(-1)})\, a_{(0)}\cdot b_{(0)},
\end{align}
$a,b\in A$. We shall denote by $A_{\sigma}$ this new algebra.
The algebra $A_{\sigma}$ is a left $H^{[\sigma]}$-comodule
algebra.

\medbreak

Let $H$ be a coradically graded pointed Hopf algebra with coradical
$H_0=\ku G$. Let $\psi\in Z^2(G,\ku^{\times})$. For the proof of the next result see \cite[Lemma 4.1]{GM}.
\begin{lma}\lelabel{lift-twist} There exists a Hopf 2-cocycle
$\sigma_\psi:H\otk H\to \ku$ such that for any homogeneous elements $x, y\in H$
\begin{align}\label{lift-twist-def}
\sigma_\psi(x,y)=\begin{cases}
              \psi(x,y), & \text{if } x,y\in H(0);\\
0, &\text{otherwise.}
               \end{cases}\end{align}\qed
\end{lma}

\subsection{Relative Hopf modules}

Let $H$ be a finite-dimensional Hopf algebra. Let $K$  be
a  left  $H$-comodule algebra and $L$ a  right  $H$-comodule algebra. Define the functors
$$F:{}_{L\Box_H K}\mo \to {}_L\mo^H_{\overline{K}}, \quad G:{}_L\mo^H_{\overline{K}}
\to {}_{L\Box_H K}\mo$$
as
$$F(M)= (L\otk K)\ot_{L\Box_H K} M, \quad G(N)=N^{\co H},$$
for all $M\in {}_{L\Box_H K}\mo, N\in  {}_L\mo^H_{\overline{K}}$.
\begin{thm}\thlabel{relative-hopf}\cite[Thm. 4.2]{CCMT}  If $L$ is a Hopf-Galois  extension then
the pair of functors $(F,G)$ gives an equivalence of categories.\qed
\end{thm}

\begin{lma}\lelabel{equi-rel} There is an equivalence of categories
${}_L\mo^H_{\overline{K}}\simeq  {}_K^H\mo_{\overline{L}}.$
\end{lma}
\begin{proof} Define the functor $I: {}_L\mo^H_{\overline{K}}\to {}_K^H\mo_{\overline{L}}$
 by $I(M)=M$. If $\delta:M\to M\otk H$, $\delta(m)=m\_0\ot m\_1$, $m\in M$, is the comodule structure then the left $H$-comodule
structure on $I(M)$ is given by
$$\widehat{\delta}: M\to H\otk M, \quad \widehat{\delta}(m)= \Ss^{-1}(m\_1)\ot m\_0,$$
for all $m\in M$. It is not difficult to prove that this functor is well-defined and
gives an equivalence of categories. \qed
\end{proof}

\subsection{Tensor categories, their representations  and the Brauer-Picard group }

 A \emph{tensor category} over $\ku$ is a $\ku$-linear Abelian rigid monoidal category
with $\ku$-bilinear tensor product.
Hereafter all tensor categories will be assumed to be over a field $\ku$. A
\emph{finite  category}  is an Abelian $\ku$-linear category such that it is equivalent
to the category of finite-dimensional representations of a finite-dimensional $\ku$-algebra.
A  \emph{finite tensor category} \cite{eo} is a tensor category with finite underlying Abelian
category such that the unit object is simple. All  functors will be assumed to be
$\ku$-linear and all categories will be finite.
\medbreak

If $\ca$ is a tensor category, we shall denote by $\ca^{\rev}$ the tensor category whose
underlying Abelian  category is $\ca$ and  the reversed  tensor product:
$X\ot^{\rev} Y= Y\ot X, \quad X, Y\in \ca.$
The associativity of $\ca^{\rev}$ is given by
$a^{\rev}_{X,Y,Z}=a^{-1}_{Z,Y,X}$ for $X, Y, Z\in \ca.$

\medbreak

For the definition of left and right module categories over a tensor category we refer
to \cite{eo}.
Let $\ca, \Do$ be finite tensor categories. For the definition of a $(\ca, \Do)$-\emph{bimodule category}
we refer to \cite{Gr}, \cite{ENO}. In few words a $(\ca, \Do)$-bimodule category
is the same as left $\ca\boxtimes \Do^{\rev}$-module category.  Here $\boxtimes$ 
denotes the Deligne tensor product of two finite abelian categories. 

A $(\ca, \Do)$-bimodule category is \emph{decomposable} if it is the direct sum of two non-trivial
$(\ca, \Do)$-bimodule categories. A $(\ca, \Do)$-bimodule category is  \emph{indecomposable}
if it is not decomposable. A $(\ca, \Do)$-bimodule category is \emph{exact}
if it is exact as a left $\ca\boxtimes \Do^{\rev}$-module category,  \cite{ENO}, \cite{Gr}.
\medbreak

 If $\ca_1, \ca_2, \ca_3$ are tensor categories and $\mo$ is a $(\ca_1,\ca_2)$-bimodule
 category and $\no$ is a $(\ca_2,\ca_3)$-bimodule
category,  the tensor product over $\ca_2$ is denoted by $\mo\boxtimes_{\ca_2}\no$. 
This category is a $(\ca_1,\ca_3)$-bimodule category. For more details on the tensor product of module categories 
the reader is referred to \cite{ENO}, \cite{Gr}.

\medbreak

If $\mo$ is a  right $\ca$-module category then $\mo^{\op}$ denotes the
 opposite  Abelian category with left $\ca$ action $\ca\times\mo^{\op} \to \mo^{\op}$, 
$ (X,M)\mapsto  M\otb X^*$ and associativity
 isomorphisms $m^{\op}_{X,Y,M}= m_{M,Y^*, X^*}$ for all $X, Y\in \ca, M\in \mo$. Similarly, if  $\mo$ is a
   left $\ca$-module category. If  $\mo$ is a $(\ca,\Do)$-bimodule category then $\mo^{\op}$ is a
$(\Do,\ca)$-bimodule category. See \cite[Prop. 2.15]{Gr}.

\medbreak

A $(\ca, \Do)$-bimodule category $\mo$ is called \emph{invertible} \cite{ENO} if
there are  equivalences of bimodule categories
$$\mo^{\op}\boxtimes_{\ca} \mo\simeq \Do, \quad \mo\boxtimes_{\Do} \mo^{\op}\simeq \ca.$$

Tensor categories $\ca$ and $\Do$ are said to be {\em Morita equivalent} if  there exists an indecomposable exact left
$\ca$-module category $\mo$ and a tensor equivalence $\Do^{\rev} \simeq  \End_\ca(\mo)$.
\medbreak
The following result seems to be well-known.
\begin{lma}\label{m-inv} Let $\ca, \Do$  be tensor categories. The following statements are
equivalent.
\begin{itemize}
 \item[1.] The categories $\ca$ and $\Do$ are Morita equivalent;

 \item[2.] there exists an invertible $(\ca, \Do)$-bimodule category.

\end{itemize}

\end{lma}

\begin{proof} (2) implies (1) is part of \cite[Proposition 4.2]{ENO}.
 Now, let us assume that the tensor categories $\ca,$ $\Do$ are Morita equivalent. Let
$\Phi:\Do^{\rev} \xrightarrow{\; \simeq  \;} \End_\ca(\mo)$ be a tensor
equivalence. Since $\mo$ is an
 indecomposable exact left $\End_\ca(\mo)$-module category then it is
an indecomposable exact right $\Do$-module category. The right $\Do$-action is
given as follows:
$$ \mo \times\Do\to \mo, \quad  M\otb Y= \Phi(Y)(M),$$
for all $M\in \mo, Y\in \Do$.
 It is easy to prove that
$\mo$ is an exact $(\ca,\Do)$-bimodule category. The functor $\Phi$ is an equivalence
of $(\Do,\Do)$-bimodule categories. Thus $\mo$ is an invertible $(\ca, \Do)$-bimodule category.
\qed\end{proof}

Given a finite tensor category $\ca$, the \emph{Brauer-Picard group} $\text{BrPic}(\ca)$
 of $\ca$ \cite{ENO} is the group of equivalence classes
of invertible exact $\ca$-bimodule categories. This group does not
depend on the Morita class of the tensor category; if $\Do$ is another tensor
category Morita equivalent to $\ca$ there is an isomorphism
$\text{BrPic}(\ca)\simeq \text{BrPic}(\Do)$. Let us explain this isomorphism. Let
$\mo$ be an invertible $(\ca, \Do)$-bimodule category. Define
\begin{equation}\eqlabel{iso-brpic-dual} \Phi: \text{BrPic}(\ca)\to \text{BrPic}(\Do), \Phi([\no])= [\mo^{\op}
\boxtimes_{\ca}\no\boxtimes_{\ca}\mo],
\end{equation}
for all $[\no]\in  \text{BrPic}(\ca)$. Here $[\no]$ denotes the equivalence
class of the module category $\no$.

%\medbreak

\subsection{ Generating some invertible bimodule categories } \selabel{autoeq}

There is a natural way to construct invertible $\ca$-bimodule categories
using tensor autoequivalences. See for example \cite{ENO}, \cite{Ga}.
Let $\ca, \Do$ be finite tensor categories and   $(F,\xi):\ca\to \ca$, $(G,\zeta):\Do\to \Do$
be tensor autoequivalences. If $\mo$ is a $(\ca, \Do)$-bimodule category
we shall denote by ${}^F\mo^G$ the following $(\ca, \Do)$-bimodule category.
The category ${}^F\mo^G$ has underlying Abelian category equal to $\mo$.
The left and right actions are given by
$$X\otb M= F(X)\otb M, \quad M \otb Y = M\otb G(Y),  $$
for all $X\in\ca, Y\in \Do$, $M\in\mo$. The left and right associativity
$$m^F_{X,Y,M}= m^l_{X,Y,M} (\xi_{X,Y}\ot\id_M),\quad m^G_{M,X,Y}=m^r_{M,X,Y} (\id_M\ot \zeta_{X,Y}). $$
Here $m^l$ (resp.  $m^r$) is the left (resp. the right)  associativity constraint of $\mo$.
If $G$ is the identity functor we shall denote ${}^F\mo^G$ simply by
${}^F\mo$ and if $F$ is the identity  we shall denote ${}^F\mo^G$  by
$\mo^G$. Let $\Aut(\ca)$ denote the group of tensor autoequivalences of $\ca$.

\begin{lma}\lelabel{bimod-ex} Let $F,G\in \Aut(\ca)$ and  let  
$\mo$ be a $\ca$-bimodule category. The following statements hold.
\begin{itemize}
 \item[1.] There are equivalences of bimodule categories
$\ca^F\boxtimes_{\ca} \ca^G\simeq \ca^{FG}$. In particular $\ca^F$ is invertible.
\item[2.] There are equivalences of bimodule categories
$$\mo
\boxtimes_{\ca} \ca^ F\simeq \mo^F, \quad \big(\mo^{\op} \big)^F\simeq \big({}^F\mo \big)^{\op}. $$

\end{itemize}

\end{lma}
\begin{proof} (1) This is statement \cite[Lemma 6.1]{Ga}. (2)
The proof of the first equivalence goes \emph{mutatis mutandis} as
the one of \cite[Prop. 3.15]{Gr}. The second equivalence is straightforward.\qed
\end{proof}

\subsection{Tensor product of bimodule categories over Hopf algebras}

Let $A, B$ be  finite-dimensional Hopf algebras. A $(\Rep(B), \Rep(A))$-bimo\-dule
 category is the same as a left $\Rep( B\otk A^{\cop})$-module category.  By \cite[Theorem 3.3]{AM} 
we know that  any exact indecomposable $(\Rep(B), \Rep(A))$-bimodule category is equivalent
to the category ${}_S\mo$ of finite-dimensional left $S$-modules,
where $S$ is a finite-dimensional  right $B\otk A^{\cop}$-simple left
 $B\otk A^{\cop}$-comodule algebra.

\begin{rem}\rmlabel{diag-coaction} The identity object in $\brp(\Rep(A)) $ is the class of the
$\Rep(A)$-bimodule category ${}_{\diag(A)}\mo$, where
 $\diag(A)=A$ as algebras, and the left
$A\otk A^{\cop}$-comodule
structure is given by:
$$\lambda:\diag(A)\to A\otk A^{\cop}\otk \diag(A), \quad \lambda(a)=a\_1\ot a\_3\ot a\_2,\; a\in A.$$
\end{rem}

We proceed to determine the tensor product over $\Rep(B)$ of a $(\Rep(A), \Rep(B))$-bimodule category 
and a $(\Rep(B), \Rep(A))$-bimodule category, both exact indecomposable. Throughout, for such a product 
we shall shortly say {\em tensor product of bimodule categories over a Hopf algebra}.

Define $\pi_A:A\ot B\to A$, $\pi_B:A\ot B\to B$  the algebra maps
$$\pi_A(x\ot y)=\epsilon(y) x, \quad \pi_B(x\ot y)=\epsilon(x) y,$$
for all $x\in A, y\in B.$

Let $K$ be a right $B\ot A^{\cop}$-simple left $B\ot A^{\cop}$-comodule algebra and
$L$ a right $A\ot B^{\cop}$-simple left $A\ot B^{\cop}$-comodule algebra. Thus
the category ${}_K\mo$  is a $(\Rep(B), \Rep(A))$-bimodule category and
  ${}_L\mo$ is a  $(\Rep(A), \Rep(B))$-bimodule category.

Recall that $\overline{L}$ is the  left $B\ot A^{\cop}$-comodule algebra
with opposite algebra structure  $L^{\op}$ and  left $B\ot A^{\cop}$-comodule
structure:
\begin{equation}\label{comod-op} \overline{\lambda}:L\to B\otk  A^{\cop}\otk L,\quad
 l\mapsto  (\Ss^{-1}_B\ot \Ss_A)(l\_{-1})\ot l\_0,
\end{equation}
for all $l\in L$.

We denote by ${}_{K}^B\mo_ {\overline{L}}$ the category
of $(K,\overline{L})$-bimodules and left $B$-comodules such that the comodule
structure is a bimodule morphism. See
 \cite[Section 3]{M3}. It has a structure of $\Rep(A)$-bimodule category.

Also $L$ is a right $B$-comodule  and $K$ is a left $B$-comodule with comodule maps given by
\begin{equation}\eqlabel{comod-op2} l\mapsto l\_0\ot \pi_B(l\_{-1}),\quad k\mapsto  \pi_B(k\_{-1})\ot k\_0,
\end{equation}
 for all $l\in L$, $k\in K$. Using this structure we can form the cotensor product $L
\Box_B K$. Define
\begin{equation}\eqlabel{coact-coprodu-t} \lambda(l\ot k)=\pi_A(l\_{-1})\ot \pi_A(k\_{-1}) \ot l\_0\ot k\_0,
\end{equation}
for all $l\ot k\in L \Box_B K$. Then
$L \Box_B K$ is a left $A\otk A^{\cop}$-comodule algebra. See \cite[Lemma 3.6]{M3}.

\medbreak

Recall that in \cite{M3} we have defined a structure of $\Rep(A)$-bimodule category
on  ${}_{K}^B\mo_{\overline{L}}$. Similarly, we can define a
of $\Rep(A)$-bimodule category structure on ${}_L\mo^B_{\overline{K}}$.

\begin{thm}\thlabel{tensorpb}
 \begin{itemize}
  \item[(a)] There is a  $\Rep(A)$-bimodule equivalence:
$$ {}_L\mo \boxtimes_{\Rep(B)}
 {}_K\mo \simeq  {}_{K}^B\mo_{\overline{L}}.$$
\item[(b)]  The functor $I$ from \leref{equi-rel} is an equivalence of $\Rep(A)$-bimodule
categories.
\item[(c)] If $L$ is a Hopf-Galois extension, as a  right  $B$-comodule algebra, then there is an
equivalence of $\Rep(A)$-bimodule categories
$${}_{L \Box_B K}\mo \simeq  {}_L\mo \boxtimes_{\Rep(B)}
 {}_K\mo.$$
\end{itemize}
\end{thm}
\begin{proof} Item (a) was proven in  \cite{M3}. Item (b)
is straightforward and (c) follows from \thref{relative-hopf} and items (a) and (b).

\end{proof}

\section{Bicategories and tensor categories}\selabel{bicat-tensc}

  For a review on basic notions on bicategories
we refer to \cite{Be, Bo}.  When the associativity constraint of a bicategory is the identity, 
the bicategory is said to be a 2-category.  
 For completeness we add that a 2-equivalence between 2-categories $\B$
and $\B'$ is a pseudo-functor $(\Theta, \theta): \B \to\B'$ such that there
is another pseudo-functor $(\Pi, \pi): \B'\to\B$ and two pseudo-natural isomorphisms
$\sigma: (\Pi, \pi) \circ (\Theta, \theta) \to Id_{\B}$ and $\tau: (\Theta, \theta) \circ (\Pi, \pi) \to Id_{B'}$.

\medbreak

It is well-known that any monoidal category $\ca$ gives rise to a bicategory
with only one object. We shall denote by $\cab$ this bicategory. If $\ca, \Do$
are strict monoidal categories, a pseudo-functor $(F,\xi):\cab\to \underline{\Do}$
is nothing but a monoidal functor between $\ca$ and $\Do$. If
$(F,\xi), (G,\zeta): \ca\to \Do$ are monoidal functors between two strict monoidal categories, a pseudo-natural transformation
between them is a pair $(\eta, \eta_0):(F,\xi) \to (G,\zeta)$ where $\eta_0\in \Do$ is an object
and for any $X\in \ca$ natural transformations
$$\eta_X: F(X)\ot \eta_0\to \eta_0 \ot G(X),$$
such that
\begin{equation}\eqlabel{pseudo-nat-m} 
(\id_{\eta_0}\ot \zeta_{X,Y})\eta_{X\ot Y}=(\eta_X\ot \id_{G(Y)})
(\id_{F(X)}\ot \eta_{ Y})  (\xi_{X,Y}\ot \id_{\eta_0})
\end{equation}
 Set $\tilde\xi_X$ for the natural isomorphism $\xi_{1,X}: F(X)\to F(X)$ (and similarly for $\zeta_{1,X}$). 
Then  $\eta_{1_{\ca}}$ is a morphism $\eta_1:\eta_0\to\eta_0$ in $\Do$ satisfying 
$$(\id_{\eta_0}\ot\tilde\zeta_X)\eta_X=(\eta_1\ot\id_{G(X)})\eta_X(\tilde\xi_X\ot\id_{\eta_0}).$$
for every $X\in\ca$. 
Given two pseudo-natural transformations $(\eta, \eta_0):(F,\xi) \to (G,\zeta)$ and $(\sigma, \sigma_0):(G,\zeta) \to (H,\chi)$ 
their composition is given by $((\id_{\eta_0}\ot\sigma)(\eta\ot\id_{ \sigma_0}), \eta_0\ot\sigma_0):(F,\xi) \to (H,\chi)$. 
A pair $(\eta, \eta_0)$ is a pseudo-natural isomorphism if there exists a pseudo-natural transformation $(\sigma, \sigma_0)$ 
such that $(\eta, \eta_0)(\sigma, \sigma_0)=(\id_{F}, 1_{\Do})$ and $(\sigma, \sigma_0)(\eta, \eta_0)=(\id_{G},1_{\Do})$. 
Consequently, 
the object $\eta_0$ is invertible in $\Do$, that is, there exists an object 
$\overline{\eta_0}\in\Do$ such that $\eta_0\ot\overline{\eta_0}= 1_{\Do}=\overline{\eta_0}\ot\eta_0$. 
Any natural monoidal transformation $\mu:(F,\xi) \to (G,\zeta)$ gives rise to a
pseudo-natural transformation $(\mu, \uno)$. 

\begin{lma}\lelabel{pseudo-vs-nat} 
Let $(\Do, c)$ be a strict braided  monoidal category. Then any pseudo-natural isomorphism
$(\eta, \eta_0):(F,\xi) \to (G,\zeta)$ between two  monoidal  functors as above produces a natural
monoidal isomorphism.
\end{lma}

\begin{proof} Given a pseudo-natural isomorphism
$(\eta, \eta_0):(F,\xi) \to (G,\zeta)$ define $\mu:(F,\xi) \to (G,\zeta)$ as
the composition
$$F(X) \xrightarrow{= } F(X) \ot \eta_0\ot \overline{ \eta_0}   \xrightarrow{\eta_X\ot \id}  \eta_0\ot G(X)\ot
\overline{ \eta_0} \xrightarrow{c_{\eta_0, G(X)}\ot \id  } G(X)\ot \eta_0\ot \overline{ \eta_0} \xrightarrow{=} G(X) $$
for any $X\in\ca$. Here $ \overline{ \eta_0}$ is the inverse objet of $\eta_0$. 
Then $\mu$ is clearly a natural transformation. We prove that it is monoidal: 
\begin{align*}
\zeta_{X,Y}&\mu(X\ots Y)=\left( (\zeta_{X,Y}\ots\id_{\eta_0})(c_{\eta_0, G(X\ot Y)}\small{\circ}\eta_{X\ot Y}) \ots \id_{\overline{\eta_0}}\right) 
(\id_{F(X\ot Y)}\ot\eta_0\ot\overline{\eta_0})\\
&=\left((\id_{G(X)}\ot c_{\eta_0, G(Y)})(c_{\eta_0, G(X)}\ots\id_{G(Y)})(\id_{\eta_0}\ot\zeta_{X,Y})\eta_{X\ot Y}\ots\id_{\overline{\eta_0}}\right) 
(\id_{F(X\ot Y)}\ot\eta_0\ot\overline{\eta_0}) %\\
\end{align*}  \vspace{-0,6cm}
\begin{equation*}
\begin{split}  %\hspace{-4cm}
%     \left
=\Big(\hspace{-0,1cm} (\id_{G(X)}\ot\id_{\eta_0\ot\overline{\eta_0}} \ot  & c_{\eta_0, G(Y)})
 (c_{\eta_0, G(X)}\ots\id_{\overline{\eta_0}\ot\eta_0} \ot\id_{G(Y)})  
(\eta_X \ots \id_{\overline{\eta_0}\ot\eta_0} \ots \id_{G(Y)})  \\  
&\quad
(\id_{F(X)}\ot \eta_0\ot\overline{\eta_0}\ot\eta_Y) (\xi_{X,Y}\ots\id_{\eta_0})
\ot\id_{\overline{\eta_0}} \hspace{-0,1cm}\Big) %\right) 
(\id_{F(X\ot Y)}\ot\eta_0\ot\overline{\eta_0})
\end{split}
\end{equation*}\vspace{-0,6cm}
\begin{align*}
= \Big(\hspace{-0,1cm} \left((c_{\eta_0, G(X)}\ots\eta_X)\ot\id_{\overline{\eta_0}}\right)&(\id_{F(X)}\ot\eta_0\ot\overline{\eta_0})\ot 
 ((c_{\eta_0, G(Y)}\small{\circ}\eta_Y)\ot\id_{\overline{\eta_0}})(\id_{F(Y)}\ot\eta_0\ot\overline{\eta_0})    \hspace{-0,1cm}\Big) \xi_{X,Y}\\
&=(\mu(X)\ots\mu(Y))\xi_{X,Y}.
\end{align*}
The second equality holds by naturality
 of the braiding and the third one is due to \equref{pseudo-nat-m} and 
because $\eta_0\ots\overline{\eta_0}=\uno$. 
\qed\end{proof}

\begin{rem}
In \leref{pseudo-vs-nat} the category $\D$ need not necessarily be braided. 
Observe that it suffices that the object $\eta_0$ is equipped 
with a lift to the Drinfel'd center $\Z(\D)$ of $\D$. 
That is, that there exists a family of natural isomorphisms 
$c_{\eta_0, X}: \eta_0\ot X\to X\ot\eta_0$ satisfying
 $c_{\eta_0, X\ot Y}=(X\ot c_{\eta_0, Y})(c_{\eta_0, X}\ot Y)$, for all $X,Y\in\D$. 
\end{rem}

Let $\C$ be a finite tensor category. We denote by $\C$-\ul{Mod} the 2-category (of $\C$-module categories) whose 0-cells
are $\C$-module categories, 1-cells are $\C$-module functors between them and 2-cells are natural transformations between
such functors (i.e. for two 0-cells $\M, \no$ there is a category $\Fun_{\C}(\M, \no)$ whose objects and morphisms present
the 1- and 2-cells of $\C$-\ul{Mod}). A $\C$-module functor $\F:\M\to\no$ is equipped with a natural isomorphism
$c_{X,M}: \F(X\crta\ot M)\to X\crta\ot\F(M)$ for $X\in\C, M\in\M$.
Then a natural transformation between two such functors $(\F, c)$ and $(\G, d)$ is $\alpha:\F\to\G$ such that
$(X\crta\ot\alpha(M))c_{X,M}=d_{X,M}\alpha(X\crta\ot M)$.

\begin{rem} \rmlabel{asoc action}
For a left $\C$-module category $\M$ the action functor $\crta\ot$ is   biexact.  
Therefore, for a $\C\x\D$-bimodule category $\M$ and $\no\in\D\x\Mod$ we have canonical isomorphisms
$X\crta\ot(M\bt_{\D}N)\iso (X\crta\ot M)\bt_{\D}N$ for all $X\in\C, M\in\M, N\in\no$.
\end{rem}

 The following result seems to be well-known. We write the proof for the reader's sake. 

\begin{thm}
Two finite tensor categories $\C$ and $\D$ are Morita equivalent if and only if there is a 2-equivalence
$(\HH, \ch): \D\x\ul\Mod \to \C\x\ul\Mod$.
\end{thm}

\begin{proof}
Let $(\HH, \ch): \D\x\ul\Mod \to \C\x\ul\Mod$ be a 2-equivalence.
For two 0-cells $\no, \Ll$ we have an equivalence functor $\HH_{\no, \Ll}: \Fun_{\D}(\no, \Ll)\to \Fun_{\C}(\HH(\no), \HH(\Ll))$
equipped with a monoidal structure $\ch$ for the composition of 1-cells. Let $\no, \Ll, \Pp\in\D\x\Mod$ and let
$\F: \no\to\Ll$ and $\G:\Ll\to\Pp$ be two 1-cells. There is a natural isomorphism
\begin{equation} \eqlabel{h monoidal}
\ch^{\no, \Ll, \Pp}: \HH_{\no, \Pp}(\F\crta\circ\G)\to\HH_{\no, \Ll}(\F)\crta\circ\HH_{\Ll, \Pp}(\G)
\end{equation}
(we usually write these in the reverted order of $\F$ and $\G$). Then $\End_{\C}(\HH(\D))= \Fun_{\C}(\HH(\D),\\ \HH(\D)) 
\simeq \Fun_{\D}(\D,  \D)\simeq\D$ as monoidal categories, thus $\C$ and $\D$ are Morita equivalent.

If $\C$ and $\D$ are Morita equivalent, by Lemma \ref{m-inv}, there is an invertible $(\C, \D)$-bimodule category $\M$ which gives rise to the desired
2-equivalence functor $(\HH, \ch): \D\x\ul\Mod \to \C\x\ul\Mod$.
On 0-cells define $\HH=\M\bt_{\D}-: \D\x\Mod\to\C\x\Mod$, that is $\HH(\no)=\M\bt_{\D}\no$ for $\no\in\D\x\Mod$.
For two 0-cells $\no, \Ll$ define the functor $\HH_{\no, \Ll}: \Fun_{\D}(\no, \Ll)\to \Fun_{\C}(\HH(\no), \HH(\Ll))$ by
$\HH_{\no, \Ll}=\M\bt_{\D}-$, i.e. for a $\D$-module functor $\F:\no\to\Ll$ we have $\HH_{\no, \Ll}(\F)=\M\bt_{\D}\F:
\M\bt_{\D}\no \to \M\bt_{\D}\Ll$. For objects $M\in\M, N\in\no$ it is $(\M\bt_{\D}\F)(M\bt_{\D} N)=M\bt_{\D}\F(N)$.
If $\F$ is a left $\D$-module functor then $\M\bt_{\D}\F$ is a left $\C$-module functor with the canonical isomorphism:
$\tilde c_{Y,M\bt_{\D}N}: (\M\bt_{\D}\F)(Y\crta\ot (M\bt_{\D}N))\to Y\crta\ot(\M\bt_{\D}\F)(M\bt_{\D}N)=
Y\crta\ot(M\bt_{\D}\F(N))$ where $Y\in\C$ and $M\in\M, N\in\no$. Due to \rmref{asoc action} the source object of
$\tilde c_{Y,M\bt_{\D}N}$ is isomorphic to $(Y\crta\ot M)\bt_{\D}\F(N)$ which is clearly isomorphic to the target object.
For two 1-cells $\F, \G: \no\to\Ll$ and a 2-cell $\alpha:\F\to\G$ we define $\HH_{\no, \Ll}(\alpha)=\M\bt_{\D}\alpha:
\M\bt_{\D}\F \to \M\bt_{\D}\G$ by
%Observe that any object in $\C\x\Mod$
%is isomorphic to $\M\bt_{\D}\Ll$ for some $\Ll\in\D\x\Mod$, since $\HH$ is a bijection. Then it suffices to define
$(\M\bt_{\D}\alpha)(M\bt_{\D} N)=M\bt_{\D}\alpha(N): M\bt_{\D}\F(N)\to M\bt_{\D}\G(N)$. The natural transformation
$\M\bt_{\D}\alpha$ fulfills the compatibility condition $(Y\crta\ot(\M\bt_{\D}\alpha)(M\bt_{\D} N))\tilde c_{Y,M\bt_{\D} N}=
\tilde d_{Y,M\bt_{\D} N}(\M\bt_{\D}\alpha)(Y\crta\ot (M\bt_{\D} N))$ by \rmref{asoc action} and since $\tilde c_{Y,M\bt_{\D} N}$
and $\tilde d_{Y,M\bt_{\D} N}$ are canonical isomorphisms.

Observe that a monoidal structure $\ch$ for the composition of 1-cells \equref{h monoidal} in this case is
an isomorphism from $\M\bt_{\D}(\F\crta\circ\G)$ to $(\M\bt_{\D}\F)\crta\circ(\M\bt_{\D}\G)$. Though, these two functors are equal,
so we take $\ch^{\no, \Ll, \Pp}$ to be the identity for all $\no, \Ll, \Pp\in\D\x\Mod$.

Since $\M$ is invertible, the pseudo-functor $(\HH, \ch)$ is a 2-equivalence.
\qed\end{proof}

\section{Bi-Galois objects and invertible bimodule categories}\selabel{bigal-invt-b}

Let $H, L$ be finite-dimensional Hopf algebras. An $(H,L)$-\emph{biGalois object},
introduced by Schauenburg in \cite{S2}, is an algebra $A$  that  is a
left $H$-Galois extension and a right $L$-Galois extension  of $\ku$  such that
the two comodule structures make it an $(H,L)$-bicomodule. Two biGalois
objects are isomorphic if there exists a bijective bicomodule morphism that is 
also an algebra map.
Denote by $\biga(H)$ the set of isomorphism classes of $(H,H)$-biGalois extensions. It is
a group with product given by $\Box_H$.

\medbreak

If $A$ is an $(H,L)$-biGalois object then the functor
$\Fc_A: \Comod(L)\to  \Comod(H)$, $\Fc_A(X)= A\Box_L X$, $X\in  \Comod(L)$, 
is a tensor  equivalence  functor \cite{U}. The tensor structure on $\Fc_A$ is as follows.
If $X, Y\in \Comod(L)$ then
\begin{equation}\eqlabel{iso-gal}
 \xi^A_{X,Y}:(A\Box_L X)\otk (A\Box_L Y)\to A\Box_L ( X\otk Y),\;  \xi^A_{X,Y}(a_i \ot x_i \ot b_j\ot y_j)
= a_i b_j \ot x_i\ot y_j
\end{equation}
for any $a_i \ot x_i \in  A\Box_L X, $ $b_j\ot y_j \in A\Box_L Y$. If $A,B$ are $(H,L)$-biGalois
objects then there is a natural monoidal isomorphism between the
tensor functors $\Fc_A, \Fc_B$ if and only if $A\simeq B$ as biGalois objects,  
 \cite[Corollary 5.7]{S2}.

\begin{lma}\lelabel{inv-examples} Let $A$ be an $(H,H)$-biGalois object.
\begin{itemize}
 \item[1.] The category
${}_A\mo$  is an invertible $\Rep(H)$-bimodule category.
\item[2.] $\Comod(H)^{\Fc_A}$ is an invertible
$\Comod(H)$-bimodule category.
\end{itemize}

\end{lma}
\begin{proof}  (1) is a consequence of \thref{tensorpb} (c) and (2)
is a particular case of \leref{bimod-ex}. \qed
\end{proof}

Bimodule categories in \leref{inv-examples} are related via the isomorphism
presented in  \equref{iso-brpic-dual}. Let us explain this assertion in detail. The category
of finite-dimensional vector spaces $vect_\ku$
is an invertible $(\Comod(H), \Rep(H^{\op}))$-bimodule category.
Let us denote
\begin{equation}\eqlabel{map-phi} \Phi: \brp(\Comod(H))\to \brp(\Rep(H^{\op}))
\end{equation}
the isomorphism described in \equref{iso-brpic-dual} using $\mo=vect_\ku$. 

\begin{prop}\prlabel{p-gal} Let $A$ be an $(H,H)$-biGalois object, then
$\Phi([\Comod(H)^{\Fc_A}])=[\mo_A]$.
\end{prop}

\begin{proof} Let us denote $\ca=\Comod(H)$. By definition we get
\begin{align*}\Phi([\ca^{\Fc_A}])
&= [vect_\ku^{op}\boxtimes_\ca\ca^{\Fc_A}\boxtimes_\ca vect_\ku]\\
&=[\big(vect_\ku^{op}\big)^{\Fc_A}\boxtimes_\ca vect_\ku]\\
&=[\big({}^{\Fc_A}vect_\ku\big)^{op}\boxtimes_\ca vect_\ku]\\
&=[\Hom_{\ca}({}^{\Fc_A}vect_\ku, vect_\ku)].
\end{align*}
The second and third equality follow from \leref{bimod-ex} (2).
The last equality is \cite[Thm. 3.20]{Gr}. It remains to prove that there is an
equivalence of bimodule categories
$$\mo_A\simeq \Hom_{\ca}({}^{\Fc_A}vect_\ku, vect_\ku).$$ We shall only sketch the proof.
Given an object $(U, \mu)\in \mo_A$, $\mu:U\otk A\to U$, define the functor $(G,c)\in
\Hom_{\ca}({}^{\Fc_A}vect_\ku, vect_\ku)$ as follows. For any $M\in vect_\ku$ 
set $G(M)=U\otk M,  $
and $c_{X,M}:U\otk  (A \Box_H X) \otk M\to X\otk U\otk M $ is
$$c_{X,M}(u\ot  a_i \ot x_i\ot  m)=x_i\ot  u\cdot a_i\ot m,$$
for all $u\in U$, $m \in M$, $\sum a_i \ot x_i\in A \Box_H X$. Conversely, given
a module functor $(G,c)\in
\Hom_{\ca}({}^{\Fc_A}vect_\ku, vect_\ku)$, since it is exact, there exists an object
$U\in vect_\ku$ such that $G(M)=U\otk M  $ for any $M\in vect_\ku$. The object
$U$ has a right $A$-module structure  $\mu:U\otk A\to U$ defined by
$$\mu= (\epsilon\ot \id_U)c_{H,\ku} (\id_U\ot \rho).$$
Here $\rho:A\to A\otk H$ is the right $H$-comodule structure. Both constructions
are well-defined and inverse of each other. \qed
\end{proof}

If $(A, \lambda)$ is a left $H$-comodule algebra and $g\in G(H)$ is a group-like
element we can define a new  comodule algebra $A^g$ on
the same underlying algebra $A$ with the coaction given by $\lambda^g: A^g \to H\otk A^g$:  
\begin{equation} \eqlabel{g-twisted}
\lambda^g(a)=g^{-1}a\_{-1}g\ot a\_0
\end{equation} 
for all $a\in A$. If $A$ is an  $(H,H)$-biGalois object let $A^g$ denote the above left comodule 
algebra whose right comodule structure remains unchanged. 

\begin{lma} $A^g$ is an $(H,H)$-biGalois object.\qed
\end{lma}

\begin{defn}\delabel{eq-bigal} If $A, B\in \biga(H)$ we shall say that $A$ is \emph{equivalent} to $B$, and denote it by $A\sim
B$ if there exists an element $g\in G(H)$ such that $A^g\simeq B$ as biGalois objects.
\end{defn}

\begin{thm}\thlabel{bigal-pseudonat} Let $A, B\in \biga(H)$. The following statements are equivalent.
\begin{enumerate}
 \item $A\sim B$;

\item there exists an equivalence  $\Comod(H)^{\Fc_A}
\simeq \Comod(H)^{\Fc_B}$ of $\Comod(H)$-bimodule categories;

\item there exists an equivalence 
${}_A\mo \simeq  {}_B\mo$ of $\Rep(H)$-bimodule categories; 

\item there exists a pseudo-natural isomorphism $(\eta, \eta_0):\Fc_A \to \Fc_B$.

\end{enumerate}

\end{thm}
\begin{proof} The equivalence between (2) and (4) is given in \cite[Lemma 6.1]{Ga}.
The equivalence between (2) and (3) %is done using 
 follows from   \prref{p-gal}. Let us prove that
(1) is equivalent to (4). Assume that there is a group-like element 
 $g\in G(H)$ and a bicomodule algebra isomorphism $f:A^g\to B$. Define
$\eta_0=\ku$ with left $H$-comodule action 
$\eta_0 \to H\otk \eta_0,$ $1\to g\ot 1$, and for any $X\in \Comod(H)$
$$\eta_X: \Fc_A(X)\otk \eta_0 \to \eta_0 \otk  \Fc_B(X), \quad \eta_X(a\ot
x\ot 1)=1\ot f(a)\ot x,$$
for all $a\ot x\in \Fc_A(X)$. Since $f$ is a right $H$-comodule morphism
the map $\eta_X$ is well-defined. Let $X, Y\in \Comod(H)$, $a\ot x \in \Fc_A(X)$,
$b\ot y \in \Fc_A( Y )$, then 
\begin{align*}(\id_{\eta_0}\ot (\xi^B_{X,Y})^{-1})&(\eta_X\ot \id)(\id\ot \eta_{ Y})(a\ot x\ot b\ot y\ot 1)=
1\ot f(a)f(b)\ot x\ot y\\
&=1\ot f(ab)\ot x\ot y\\
&=\eta_{X\ot Y} ((\xi^A_{X,Y})^{-1}\ot \id_{\eta_0})(a\ot x\ot b\ot y\ot 1).
\end{align*}
Thus \equref{pseudo-nat-m} is fulfilled and $(\eta, \eta_0)$ is a pseudo-natural transformation.

Now, let us assume that there exists a 
pseudo-natural isomorphism $(\eta, \eta_0):\Fc_A \to \Fc_B$. 
Since $\eta_0 \in \Comod(H)$ is an invertible object it is one-dimensional. Hence, 
there exists a group-like element
$g\in G(H)$ such that the coaction  $\eta_0 \to H\otk \eta_0$ is given by $1\mapsto g\ot 1$.
Define $f:A\to B$ as the composition
$$A\xrightarrow{\iota }A\otk \eta_0 \xrightarrow{\rho \ot \id}A \Box_H H \otk \eta_0 \xrightarrow{\eta_H}
 \eta_0  \otk B \Box_H H \xrightarrow{\id\ot \epsilon} \eta_0  \otk B \xrightarrow{\pi} B. $$
We must show that $f$ is an algebra map and an $H$-bicomodule homomorphism. 

\begin{claim} $f:A\to B$ is an algebra map.
\end{claim}
\begin{proofc} It is enough to prove that $\eta_H$ is an algebra map. Observe that
$A\Box_H H$ is a subalgebra of $A\otk H$ and the algebra structure on $A \Box_H H \otk \eta_0$
is  that of the tensor product algebra. We shall denote 
$$m_1:\Fc_A(H) \otk \eta_0 \otk \Fc_A(H) \otk \eta_0\to \Fc_A(H) \otk \eta_0,$$
$$m_2:\eta_0\otk\Fc_B(H) \otk \eta_0\otk\Fc_B(H) \to \eta_0\otk\Fc_B(H),$$
the algebra structures. Define the isomorphisms
\begin{align*}
\gamma_0&:\eta_0\otk \Fc_B(H)\otk \Fc_B(H)\to \eta_0\otk \Fc_B(H)\otk \eta_0\otk \Fc_B(H), & 
\gamma_0(1\ot a\ot b)&=1\ot a\ot 1\ot b,\\
\gamma_1&:\Fc_A(H)\otk \Fc_A(H)\otk\eta_0\to \Fc_A(H)\otk\eta_0\otk \Fc_A(H)\otk\eta_0, 
& \gamma_1(x\ot y\ot1)&=x\ot 1\ot y\ot1,\\
\gamma_2&:\Fc_A(H)\otk\eta_0\ot \Fc_B(H)\to \Fc_A(H)\otk\eta_0\otk\eta_0\otk \Fc_B(H),
 & \gamma_2(x\ot1\ot b)&=x\ot 1\ot 1\ot b,
\end{align*}
for all $a, b\in  \Fc_B(H)$, $x, y\in  \Fc_A(H).$ It is not difficult to prove that
\begin{equation}\eqlabel{eq-dem11} 
(\id_{\Fc_A(H)}\ot \id_{\eta_0}\ot \eta_H)\gamma_1=  \gamma_2(\id_{\Fc_A(H)}\ot \eta_H),
\end{equation}
\begin{equation}\eqlabel{eq-dem12}
( \eta_H\ot \id_{\eta_0} \ot \id_{\Fc_B(H)})  \gamma_2 =  \gamma_0 ( \eta_H\ot  \id_{\Fc_B(H)}).
\end{equation}
\end{proofc}
Let us denote $m:H\otk H\to H$ the product. Since $m$ is a morphism in $\Comod(H)$,
the naturality of $\eta$ implies that
$(\id_{\eta_0}\ot \Fc_B(m))\eta_{H\ot H}=\eta_H(\Fc_A(m)\ot \id_{\eta_0}).$
The following equalities are readily verified:
\begin{equation}\eqlabel{eq-dem111}
 \Fc_A(m)\ot \id_{\eta_0} = m_1 \gamma_1 (\xi^A_{H,H}\ot \id_{\eta_0}),\quad 
 \id_{\eta_0}\ot  \Fc_B(m)= m_2\gamma_0 (\id_{\eta_0}\ot \xi^B_{H,H}).
\end{equation}
Since $\eta$ is pseudo-natural, then
\begin{align}\eqlabel{eq-dem1}(\eta_H\ot \id_{ \Fc_B(H)})(\id_{ \Fc_A(H)}\ot\eta_H)(\xi^A_{H,H}\ot \id_{\eta_0})
=(\id_{\eta_0}\ot \xi^B_{H,H})\eta_{H\ot H}.
\end{align}
Let us denote the isomorphism $\phi=\gamma_1(\xi^A_{H,H}\ot \id_{\eta_0})$. We have that
\begin{align*} (\eta_H\ot \eta_H)\phi&=(\eta_H\ot \id_{\eta_0\ot \Fc_B(H)})
(\id_{\Fc_A(H)\ot\eta_0}\ot\eta_H)\gamma_1(\xi^A_{H,H}\ot id_{\eta_0})\\
&=(\eta_H\ot \id_{\eta_0\ot \Fc_B(H)})(\gamma_2(\id_{\Fc_A(H)}\ot\eta_H))(\xi^A_{H,H}\ot \id_{\eta_0})\\
&=\gamma_0(\eta_H\ot \id_{\Fc_B(H)})(\id_{F(H)}\ot\eta_H)(\xi^A_{H,H}\ot \id_{\eta_0})\\
&=\gamma_0(id_{\eta_0}\ot\xi^B_{H,H})\eta_{H\ot H}.
\end{align*}
The second equality follows from \equref{eq-dem11}, the third equality by 
\equref{eq-dem12} and the last equality follows from \equref{eq-dem1}. Now, we have that
\begin{align*} \eta_H m_1 \phi&=  \eta_H (\Fc_A(m)\ot \id_{\eta_0})
=(\id_{\eta_0}\ot \Fc_B(m))\eta_{H\ot H}\\
&=m_2\gamma_0 (\id_{\eta_0}\ot \xi^B_{H,H})\eta_{H\ot H}=m_2(\eta_H\ot \eta_H)\phi.
\end{align*}
The first and third equalities follow from \equref{eq-dem111}. Hence
$\eta_H m_1= m_2(\eta_H\ot \eta_H)$ and $\eta_H$ is an algebra map. That 
$\eta_H(1\ot 1\ot 1)=1\ot 1\ot 1$ follows from the naturality of $\eta$ 
(since $A\Box_H \ku=\ku$, it is $\eta_\ku=\id_\ku$). \qed

For any vector space $V$ we shall denote by $V^t$ the same vector space $V$ with
trivial left $H$-coaction: $\lambda^t:V^t\to H\otk V^t$, $\lambda^t(v)=1\ot v$ for all
$v\in V$. Then $A\Box_H V^t=V^t$ and $\eta$ is additive, because for any
vector space $V$ we have that $\eta_{V^t}=\id_{V}$.

\begin{claim} $f:A\to B$ is a right $H$-comodule  map.
\end{claim}
\begin{proofc} 
The spaces  $A \Box_H H \otk \eta_0$ and 
 $\eta_0  \otk B \Box_H H$ have a right $H$-comodule structure as follows:
$$\rho_1: A \Box_H H \otk \eta_0\to  A \Box_H H \otk \eta_0\otk H, \;\; 
\rho_2: \eta_0  \otk B \Box_H H \to \eta_0  \otk B \Box_H H \otk H,$$
$$\rho_1(a\ot h\ot 1)=a\ot h\_1\ot 1\ot h\_2, \quad \rho_2(1\ot b\ot h)=1 \ot b\ot h\_1\ot h\_2,$$
for all $a\ot h\in A \Box_H H$, $b\ot h \in B \Box_H H$. With these structures, the maps 
$\iota, \rho\ot \id, \id\ot \varepsilon$ and $\pi$ are comodule morphisms. Hence,
it is enough to prove that $\eta_H$ is a right $H$-comodule map. First note that
$$\rho_1= (\id\ot \eta_{H^t}) (\xi^A_{H, H^t}\ot \id_{\eta_0}) (\id\ot \Delta\ot \id), \quad 
\rho_2=(\id_{\eta_0}\ot \xi^B_{H, H^t} ) (\id\ot \id\ot \Delta).$$
%Here $\tau: A \Box_H H \otk H\otk  \eta_0 \to A \Box_H H \otk  \eta_0\otk H$ is the usual flip. 
Now, we have
\begin{align*} (\eta_H \ot\id_H)\rho_1&=
 (\eta_H \ot\id_H) (\id\ot \eta_{H^t})(\xi^A_{H, H^t}\ot \id_{\eta_0}) (\id\ot \Delta\ot \id)\\
&= (\id_{\eta_0}\ot \xi^B_{H, H^t})  \eta_{H\otk H^t}  (\id\ot \Delta\ot \id)\\
&=(\id_{\eta_0}\ot \xi^B_{H, H^t} )(\id\ot \id\ot \Delta)\eta_H=  \rho_2 \eta_H.
\end{align*}
The %first euality follows since $\eta_{H^t}=\id_H$, the 
second equality follows from \equref{pseudo-nat-m}, and the third one follows from the naturality of 
$\eta$ since $\Delta:H\to H\otk H^t$ is a left $H$-comodule map.
\qed\end{proofc}

\begin{claim} $f:A^g\to B$ is a left $H$-comodule  map.
\end{claim}
\begin{proofc} If $\lambda:C\to H\otk C$ is a left $H$-comodule and $g\in G(H)$, define the left
$H$-comodules $C^{(g)}$ and ${}^{(g)}C$ as follows. As vector spaces $C^{(g)}={}^{(g)}C =C$,
the comodule structures $\lambda^{(g)}:C^{(g)}\to H\otk C^{(g)}$, 
${}^{(g)}\lambda:{}^{(g)}C\to H\otk {}^{(g)}C$ are defined by
$$\lambda^{(g)}(c)=c\_{-1}g\ot c\_0, \quad {}^{(g)}\lambda(c)=gc\_{-1}\ot c\_0,$$
for all $c\in C$.  Note that $f:A^g\to B$ is a left $H$-comodule  map if and only if
$f:A^{(g)}\to {}^{(g)}B$ is a left $H$-comodule  map. It is enough to observe that
$\iota:A^{(g)}\to A\otk \eta_0$ and $\pi:\eta_0 \otk B\to {}^{(g)}B$ are comodule morphisms. 
\end{proofc}\qed
\end{proof}

Define $\inbi(H)$ as the set of isomorphism classes of $(H,H)$-biGalois
objects $A$ such that $A\sim H$.

\begin{cor} \colabel{bigal seq}
There is an exact sequence of groups
$$1\to \inbi(H) \to \biga(H) \to  \brp(\Rep(H)).$$
\end{cor}
\begin{proof} Define the map $\phi: \biga(H) \to  \brp(\Rep(H))$,
$\phi([A])= [{}_A\mo]$, for any isomorphism class $[A]\in \biga(H).$
Here $[{}_A\mo]$ denotes the equivalence class of the bimodule category
${}_A\mo$. By \leref{inv-examples} it follows that $\phi$ is well-defined  and by \thref{tensorpb} 
it is a group map. 
If $\phi([A])$ is the trivial element in $\brp(\Rep(H))$,
by \thref{bigal-pseudonat} it follows that $A\sim H$. \qed
\end{proof}
\begin{rem} The above exact sequence can be thought of as
an analogue  of the sequence studied in \cite{BFRS}.
It would be interesting to give an interpretation of the results obtained in
\cite{BFRS} in this context.
\end{rem}

\begin{cor} If  $H^*$  is a quasi-triangular Hopf algebra, there is an injective group 
homomorphism
$\biga(H) \to  \brp(\Rep(H)).$
\end{cor}

\begin{proof} Let $A\in \inbi(H)$. It follows from  \thref{bigal-pseudonat} that there
is a pseudo-natural isomorphism between the monoidal functors $\Fc_A$ and $\Fc_H$.
From \leref{pseudo-vs-nat} it follows that there is a natural  monoidal  isomorphism between
$\Fc_A$ and $\Fc_H$. This implies that $A\simeq H$ as biGalois objects. \qed
\end{proof}

\section{Families of invertible $\Rep(T_q)$-bimodule categories}\selabel{taft-ex}

Let $n\geqslant 2$ be a natural number and $q$  a
$n$-th primitive root of unity. The Taft algebra is
$$T_q=\ku\langle g,x\vert g^n=1, x^n=0, gx=q\,  xg\rangle.$$
The structure of a Hopf algebra on $T_q$ is such that $g$ is group-like,
$x$ is $(1,g)$-primitive, that is, $\Delta(x)=x\ot 1+g\ot x$
with $S(x)=-g^{-1}x$. When $n=2$ note that we recover Sweedler's Hopf algebra $H_4$. The Taft algebra is
isomorphic to a Radford biproduct
\begin{equation} \eqlabel{Taft}
T_q\iso k[x]/(x^n)\# \ku \Zz_{n}
\end{equation}
(send $G\mapsto 1\ot g$ and $X\mapsto  x\ot 1 $), 
where $g\cdot x=q\,  x$. The
following technical result will be needed later.
\begin{lma}\lelabel{iso-cop} There is a Hopf algebra isomorphism
$\phi:T_{q^{-1}} \xrightarrow{\;\; \simeq\;\;}  T_q^{\cop}$.
\end{lma}
\begin{proof} Let us assume that $T_q$ is generated by elements $g,x$ such that
$g^n=1, x^n=0, gx=q\,  xg$ and
 $$\Delta(g)=g\ot g,\quad \Delta(x)=x\ot 1+g\ot x,$$
and $T_{q^{-1}}$ is generated by elements $g,y$ such that
$g^n=1, y^n=0, g^{-1}y=q^{-1}\,  yg^{-1}$ and the coproduct is determined by
 $$\Delta(g)=g\ot g,\quad \Delta(y)=y\ot 1+g^{-1}\ot y.$$
 (Strictly speaking, in $T_{q^{-1}}$ we should take $g^{-1}$ for a generator rather 
than $g$, but all the relations remain the same.) 
The algebra map $\phi: T_{q^{-1}} \to  T_q^{\cop}$ determined by
$ \phi(g)=g, \quad \phi(y)=x g^{-1},$
is a well-defined Hopf algebra isomorphism.\qed
\end{proof}

As we are interested in finding invertible $\Rep(T_q)$-bimodule categories, which are 
left $\Rep(T_q\ot T_q^{cop})$-module categories, from now on we shall consider 
the Hopf algebra $H=T_q\ot T_{q^{-1}}$.

\smallbreak

%Set  $H=T_q\ot T_{q^{-1}}$  and 
Let $V_1$ and $V_2$ be the one dimensional vector spaces spanned by
 $x$ and $y$ respectively,
$\Zz_{n}=\langle g\rangle=\langle  g^{-1} \rangle$ with $g^n=1$ and let 
$G=\Zz_{n}\times \Zz_{n}=
\langle g\rangle\times\langle g\rangle$.
 The vector spaces $V_1$ and $V_2$ are $G$-modules via
$$(g^i, g^j)\cdot x = g^i\cdot x=q^i x \quad\textnormal{and}\quad (g^i, g^j)\cdot y = g^j\cdot y=q^{j } y.$$
The algebra $H$ is generated by the elements
$\{f \in G, x, y\}$ subject  to relations
$$x^n=0=y^n, \quad xy=yx, \quad fx=(f\cdot x) f, \quad fy=(f\cdot y)f.$$
Its Hopf algebra structure is given by
$$\Delta(x)=x\ot 1 + (g,1)\ot x, \quad  \Delta(y)=y\ot 1 + (1,g^{-1})\ot y,
\quad\Delta(f)=f\ot f.$$
In other words, $H=\nic(V)\# \ku G$ where  $\nic(V)$ is the
Nichols algebra of the Yetter-Drinfeld module $V=V_1\oplus V_2$ over $\ku G$. The coaction
$\delta:V_1\oplus V_2\to \ku G\otk (V_1\oplus V_2)$ is given by
$$\delta(v)=(g,1)\ot v, \quad \delta(w)= (1,g^{-1})\ot w,$$
for all $v\in V_1, w\in V_2$. Let us define a new Hopf algebra that will be used later. 
If $\chi_1,\chi_2:G\to \ku$ are
characters then $V$ has a new action of $G$ as follows. For
any $f\in G$, $v\in V_1$, $w\in V_2$
$$f \rhd v= \chi_1(f)\; f\cdot v, \quad f \rhd w= \chi_2(f)\; f\cdot w.$$
Let $\chi_1, \chi_2$ be characters such that $\chi_1(1,g^{-1})\chi_2(g,1) =1$. In this case
$V $ with the new action and the same coaction is  a Yetter-Drinfeld module over $\ku G$ that we
 shall denote by $V_{(\chi_1,\chi_2)}$. Observe
that $V$ and also $V_{(\chi_1,\chi_2)}$ are quantum linear spaces, see \cite{AS}.

\begin{defn} If $\chi_1,\chi_2:G\to \ku$ are
characters such that $\chi_1(1,g^{-1})\chi_2(g,1)=1 $ we denote $H_{(\chi_1,\chi_2)}= \nic(V_{(\chi_1,\chi_2)})\# \ku G$.
\end{defn}

The algebra $H_{(\chi_1,\chi_2)}$ is generated by  elements
$\{f \in G, x, y\}$ subject  to relations
$$x^n=0=y^n, \quad xy=\chi_2(g,1)\; yx, \quad fx=(f\rhd x) f, \quad fy=(f\rhd y)f.$$
Its coproduct  is the same as the coproduct of $H$.

\subsection{Twisting $T_q\ot T_{q^{-1}}$}

We will next investigate the Hopf  algebra $(T_q\ot T_{q^{-1}})^{[\sigma]}$
for Hopf 2-cocycles $\sigma$ obtained as in \leref{lift-twist}. Let $\psi\in Z^2(G, \ku^{\times})$.
 We define
$\chi_1,\chi_2: G\to k^{\times}$ characters on $G$, via
\begin{equation}\eqlabel{charact-psi}
 \chi_1(f)=\frac{\psi(f, (g,1))} {\psi((g,1), f)} \quad\textnormal{and} \quad
\chi_2(f)=\frac{\psi(f, (1, g^{-1}))} {\psi((1, g^{-1}),f)}.
\end{equation}

The proof of the following result is straightforward.
\begin{prop}
Assume $\psi\in Z^2(G, k^{\times})$ is a 2-cocycle. Let  $\sigma: H\ot H\to \ku$ be a
 2-cocycle coming from $\psi$ as in  \leref{lift-twist} and $\chi_1, \chi_2$ characters
in $G$ defined in \equref{charact-psi}. There
is an isomorphism of Hopf algebras $H^{[\sigma]}\simeq H_{(\chi_1,\chi_2)}$.\qed
\end{prop}

\subsection{ Homogeneous coideal subalgebras in Taft Hopf algebra}

Due to \cite[Theorem 6.1]{Sk} any coideal subalgebra of a finite-dimensional Hopf algebra $H$ 
 is an $H$-simple comodule algebra. Its \emph{lifting} will also be of that type and thus it will 
determine an exact indecomposable $\Rep(H)$-bimodule category. Any exact indecomposable 
module category  emerges  in this way. This is why a fundamental piece of 
information needed to compute exact $\Rep(T_q)$-bimodule categories
is the classification of its coideal subalgebras. 
This is the main goal of this section. As before, we set $H=T_q\ot T_{q^{-1}}$. 

\medbreak
Note that $H(1)=(V_1\oplus V_2)\ot \ku G$. For $(v_1, v_2)=(\alpha x,\beta y)\in V_1\oplus V_2$,
with $\alpha, \beta\in \ku$, we will denote
$$[(v_1,v_2)]=v_1+v_2(g,g)\in H(1)\quad\textnormal{and}\quad
\widetilde{[(v_1,v_2)]}=v_2+v_1(g^{-1},g^{-1})\in H(1).$$

\begin{rem} \rmlabel{corchetes}
The following holds:
\begin{equation}
[(v_1,v_2)]^n=\widetilde{[(v_1,v_2)]}^n=0
\end{equation}
\begin{equation} \eqlabel{x-y}
\Delta([(v_1,v_2)])=v_1\ot 1 + v_2(g,g)\ot (g,g) + (g,1)\ot [(v_1,v_2)]
\end{equation}
\begin{equation} \eqlabel{tilde-x-y}
\Delta(\widetilde{[(v_1,v_2)]})=v_2\ot 1 + v_1(g^{-1},g^{-1})\ot (g^{-1},g^{-1}) +
(1,g^{-1})\ot \widetilde{[(v_1,v_2)]}.
\end{equation}
Observe that if $K$ is a homogeneous left coideal subalgebra of $H$
and $[(v_1,v_2)]\in K$ or $\widetilde{[(v_1,v_2)]}\in K$ where some $v_i$ is not null then, since both \equref{x-y} and \equref{tilde-x-y} are elements in $H(0)\ot K(1) \hspace{0,1cm}
\oplus\hspace{0,1cm} H(1)\ot K(0)$,  it follows that $(g, g)\in K(0)$.
\end{rem}

\begin{defn}
A  coideal subalgebra datum is a collection $(W^1, W^2, W^3, F)$ such that
\begin{enumerate}
\item
$W=W^1\oplus W^2\oplus W^3$ is a subspace of $V_1\oplus V_2$ such that 
$$W\cap V_1=W^1, \quad W\cap V_2=W^2;$$ 
\item $W^3\subseteq V_1\oplus V_2$ is a subspace such that 
$$W^3\cap W^1\oplus W^2=0,\quad W^3\cap V_1=0=W^3\cap V_2;$$
\item $F\subseteq G$ is a subgroup that leaves invariant all  subspaces $W^i$, $i=1,2,3$;
\item if $W^3\not=0$ then $(g, g)\in F$;
\end{enumerate}
We denote by $C(W^1, W^2, W^3, F)$ the subalgebra of $H$ generated by $\ku F$ and elements in $W^1\oplus W^2$
and $\{ [w], \widetilde{[w]} \hspace{0,2cm}:  w\in W^3\}$.
\end{defn}
If   $\chi_1,\chi_2:G\to \ku$ are
characters such that $\chi_1(1,g^{-1})\chi_2(g,1)=1 $ and $(W^1, W^2, W^3, F)$
is a coideal subalgebra datum, we shall denote by $C_{(\chi_1,\chi_2)}(W^1, W^2, W^3, F)$
the subalgebra of $H_{(\chi_1,\chi_2)}$ generated by $\ku F$ and elements in $W^1\oplus W^2$
and $\{ [w], \widetilde{[w]} \hspace{0,2cm}:  w\in W^3\}$.

\begin{rem}\label{remark-coi1} If $(W^1, W^2, W^3, F)$ is a coideal subalgebra datum then we conclude that
 if $W^3\neq 0$ then $W^1=W^2=0$. 
 Indeed, suppose that $W^3\neq 0$ and $W^1\neq 0$. Then $W^1\oplus W^3=W$ 
and $(W^1\oplus W^3)\cap V_2\neq 0$, 
since $V_1$ has dimension $1$. 
This implies that $W^2\neq 0$, but if $W^1\neq 0$ and $W^2\neq 0$,
 then it must be $W^3=0$, contradiction. 
\end{rem}

\begin{lma}
The algebra $C(W^1, W^2, W^3, F)$ (resp.  $C_{(\chi_1,\chi_2)}(W^1, W^2, W^3, F)$ ) is a homogeneous
left coideal subalgebra of $H$ (resp.  $H_{(\chi_1,\chi_2)}$). \qed
\end{lma}

\begin{thm}\label{class-hcoid}
Any homogeneous left coideal subalgebra $K=\oplus_{i=0}^m K(i)$ in $ H$ (resp.  $H_{(\chi_1,\chi_2)}$)
 is of the form $K=C(W^1, W^2, W^3, F)$  (resp.   $C_{(\chi_1,\chi_2)}(W^1, W^2, W^3, F)$ ),
for some coideal subalgebra datum $(W^1, W^2, W^3, F)$.
\end{thm}

\begin{proof} We shall assume that $\chi_1,\chi_2$ are trivial, since the proof for non-trivial
 characters is completely analogous.

Given that $K(0)\subseteq \ku G$ is a left coideal subalgebra, it is $K(0)=\ku F$
for some subgroup $F\subseteq G$.
If $K(1)=0$, then $K=\ku F$. Indeed, if $x\in K(2)$ then $\Delta(x)\in H(0)\ot K(2)\hspace{0,1cm} \oplus
\hspace{0,1cm} H(2)\ot K(0)$, therefore
$x\in H(0)\oplus H(1)$, and since  $(H(0)\oplus H(1))\cap H(2)=0$, it follows $x=0$. Similarly, one proves that $K(n)=0$
 for all $n$.

Suppose that $K(1)\not=0$. The vector space $K(1)$ is a $\ku G$-subcomodule of
 $(V_1\oplus V_2)\ot \ku G$ via
$$(\pi\ot \id)\Delta: K(1)\to  \ku G\ot K(1),$$
where $\pi: H\to \ku G$ is the canonical projection. We may write $K(1)=\oplus_{f\in G}K(1)_{f}$
where $K(1)_{f}= \{ x\in K(1) \vert (\pi\ot \id)\Delta(x)=f\ot x\}$. Then it is
$$K(1)_{f}\subseteq V_1\ot \ku \langle(g^{-1},1)f \rangle \hspace{0,1cm}
\oplus \hspace{0,1cm} V_2\ot  \ku\langle(1,g)f \rangle.$$
In particular we have:
$$K(1)_{(g,1)}=W^1 \oplus \widetilde{W}^2(g,g) \oplus U^3 \qquad
K(1)_{(1,g^{-1})}=W^2 \oplus \widetilde{W}^1(g^{-1},g^{-1}) \oplus \widetilde{U}^3.$$
Here $W^1$ is the intersection of $K(1)_{(g,1)}$ with $V_1$, $\widetilde{W}^2(g,g)$
 is the intersection
of $K(1)_{(g,1)}$ with $V_2\ot  \ku\langle(g,g)\rangle$, and $U^3$ is a direct complement.
 Concretely, $U^3$
is a subspace of $V_1\oplus V_2\ot \ku\langle(g,g)\rangle$ consisting of elements of the form
$[w]$, where $w\in W^3$ and $W^3\subseteq V_1\oplus V_2$. Given that
$U^3\cap W^1\oplus\widetilde{W}^2(g, g)=0$,
it follows that $W^3\cap W^1\oplus\widetilde{W}^2=0$. \medbreak

 Analogously, for $K(1)_{(1,g^{-1})}$ we have that $W^2$ is the
intersection of $K(1)_{(1,g^{-1})}$ with $V_2$, $\widetilde{W}^1(g^{-1},g^{-1})$ is the intersection of
$K(1)_{(1,g^{-1})}$ with $V_1\ot k\langle(g^{-1},g^{-1})\rangle$, and $\widetilde{U}^3$ is a direct complement. The
elements of $\widetilde{U}^3$ are of the form $\widetilde{[w]}$,
 where $w\in\widetilde{W}^3$ and  $\widetilde{W}^3\subseteq V_1\oplus V_2$.
 We next prove the following two results:

\begin{claim} \label{Claim 6.1}
%If $W^3\not=0$ or  $\tilde W^2\not=0$, then $(g^{-1},g)\in F$.
%If $\tilde W^3\not=0$ or  $\tilde W^1\not=0$, then $(g,g^{-1})\in F$.
If any of $\widetilde{W}^1, \widetilde{W}^2, \widetilde{W}^3$ or $W^3$ is different from $0$,
 then $(g,g)\in F$.  If $\widetilde{W}^1\not=0$,
 then $\widetilde{W}^1=W^1$. If $\widetilde{W}^2\not=0$, then $\widetilde{W}^2=W^2$.
\end{claim}

\begin{proof}
Take $0\not=(v_1, v_2)\in W^3$. Then $0\not=[(v_1, v_2)]\in U^3$ and by \rmref{corchetes}
 $(g, g)\in F$.
The proof is analogous if $\widetilde{W}^1, \widetilde{W}^2$ or $\widetilde{W}^3$ are different
from zero. For the second claim take $w\in\widetilde{W}^1$. Then $w(g^{-1},g^{-1})\in K(1)_{(1,g^{-1})}$ and
$w\in K(1)_{(g,1)}$ and it must be $w\in W^1$. Similarly,
one proves the other inclusion and we get $\widetilde{W}^1=W^1$. The other equality is proven analogously.
\qed\end{proof}

\begin{claim}  $K(1)=W^1 F \oplus W^2 F \oplus U^3 F$.
\end{claim}

\begin{proof}
Take $f \in G$ and $0\not=x\in K(1)_{f}$. Then for some $v_1\in V_1$ and $v_2\in V_2$ it is
$$x=v_1(g^{-1},1)f  + v_2(1,g)f$$
and
$$ \Delta(x)=v_1(g^{-1},1)f \ot (g^{-1},1)f + v_2(1,g )f \ot (1,g)f + f\ot x $$
is an element in $H(0)\ot K(1)\hspace{0,1cm} \oplus \hspace{0,1cm} H(1)\ot K(0)$.
If $v_1\not=0$, then $(g^{-1},1)f
\in F$ and hence $xf^{-1}(g,1)\in K(1)_{(g,1)}$ and $x\in K(1)_{(g,1)}F \subseteq
W^1F \oplus \widetilde{W}^2(g,g)F \oplus U^3F \subseteq W^1F \oplus W^2F \oplus U^3F$ - the
latter inclusion is due to Claim  \ref{Claim 6.1}.
If $v_1=0$, then $v_2\not=0$ and it follows $(1,g)f\in F$. Thus
$xf^{-1}(1,g^{-1})\in K(1)_{(1,g^{-1} )}$.
If $\widetilde{W}^1=\widetilde{U}^3=0$, then $xf^{-1}(1,g^{-1})\in W^2$ and $x\in W^2F$ and the claim follows.
If any of $\widetilde{W}^1$ and $\widetilde{U}^3$ is not zero, then $(g,g)\in F$.
But then $(g^{-1},1)f=(g^{-1},g^{-1}) (1,g)f\in F$ and thus $xf^{-1}(g,1)\in K(1)_{(g,1)}$.
It follows $x\in K(1)_{(g,1)}F$, which we already have proved is a subspace of $W^1F \oplus W^2F \oplus U^3F$.
\qed\end{proof}

To finish the proof of the Theorem we will prove that $K$ is generated as an algebra by $K(0)$ and
$K(1)$, which yields $K=C(W^1, W^2, W^3, F)$.
 Let $B=\{b_i \}$ be a basis of $V_1\oplus V_2$. For any
$v_1\in V_1$ and $v_2\in V_2$ it is $(v_1, 0)=\sum_i\alpha_i b_i$
and $(0, v_2)=\sum_i\beta_i b_i$ for some
$\alpha_i, \beta_i\in k$. Then $v_1=\sum_i\alpha_i [b_i]$ and
$v_2=\sum_i\beta_i [b_i](g^{-1}, g^{-1})$. So we have
that $H$ is generated as an algebra by the set
$$\{[b_i], f: \hspace{0,1cm} b_i\in B, f\in G\}.$$
Now, let $\{b_1, ..., b_r\}$ be a basis of $W=W^1\oplus W^2\oplus W^3$ and extend it to a basis
$\{b_1, ..., b_t\}$ of $V_1\oplus V_2$ with $r\leq t$. For $n>1$ an arbitrary $x\in K(n)$
 has the form
$$x=\sum_{{s_j\in\{0,1\}} \atop {f_i\in G}} \alpha_{s_1, ..., s_t,i} [b_1]^{s_1}
[b_2]^{s_2}\cdots [b_t]^{s_t} f_i$$
for some $\alpha_{s_1, ..., s_t,i}\in k$, where $s_1+ ...+ s_t=n$.
Let $p: H\to H(1)$ be the canonical projection. Then
 $$(\id\ot p)\Delta(x)=\sum_l \sum_{{s_j\in\{0,1\}} \atop {f_i\in G}} \alpha_{s_1, ..., s_t,i}
h_{s_1, ..., s_t,i,l}\ot [b_l]f_i$$ 
for some  $0\not= h_{s_1, ..., s_t,i,l}\in H(n-1)$  is an element in $H(n-1)\ot K(1)$. 
So for all $l>r$ with $s_l=1$
it must be $\alpha_{s_1, ..., s_t,i}=0$ and $K(n)$ is generated as an algebra by $K(1)$.
\qed\end{proof}

\begin{rem} A coideal subalgebra datum depends whether it is on $H$ or $H_{(\chi_1,\chi_2)}$.
Since $V_1, V_2$ are 1-dimensional vector spaces we can give a description
of all possible coideal subalgebra data. 
Let $( W^1, W^2, W^3, F)$ be a coideal subalgebra datum (either for
$H$ or $H_{(\chi_1,\chi_2)}$ ). Then $W^3$ is either
null or 1-dimensional. If $W^3\neq 0$ then, using Remark \ref{remark-coi1}
we get that $W^1= W^2=0$ and
$$( W^1, W^2, W^3)=(0,0, < \xi\, x + y>_{\ku})$$
for some $0\neq\xi\in \ku$. We call this coideal subalgebra datum
of \emph{type} $\xi$.

If $W^3=0$ then
$$( W^1, W^2, W^3)=(<\delta_1 \, x>_{\ku} , <\delta_2 \, y>_{\ku}, 0)$$
for some $\delta_1,\delta_2\in \{0,1\}$.  We call this coideal subalgebra datum
of \emph{type} $(\delta_1,\delta_2)$.
 We have already observed that if $W^3\neq 0$, then $(g,g)\in F$.

Assume  $( W^1, W^2, W^3, F)$ is a coideal subalgebra datum for $H_{(\chi_1,\chi_2)}$.
Let $f=(g^i, g^j)\in F$. Since $F$ leaves invariant
the subspace $W^3$, then
$$f\cdot (\xi\, x + y)= \xi q^i \chi_1(f)\, x + q^j \chi_2(f)\, y \in  < \xi\, x + y>_{\ku}.$$
Hence
\begin{equation}\eqlabel{coi-stab-F} q^i \chi_1(f)=q^j \chi_2(f)
\end{equation}
for all $f\in F$. In particular, if $\chi_1=\chi_2$, then $i=j$.
 Therefore $F$ contains the cyclic group generated by $(g,g)$.
\end{rem}

\subsection{ Families of $T_q\ot T_{q^{-1}}$-comodule algebras}
%In this section we shall denote $H=T_q\ot T_{q^{-1}}$. We shall give a classification
%of exact indecomposable $\Rep(T_q)$-bimodule categories.

We shall introduce families of non-equivalent right $H$-simple
left $H$-comodule algebras and \emph{a fortiori}, families of
exact  indecomposable  $\Rep(H)$-module categories, where $H=T_q\ot T_{q^{-1}}$.  
 We shall define them by generators and relations extending the information 
from the coideal subalgebras from the previous section. It will turn out that the former 
families are { liftings} of the latter.  
\medbreak

\begin{defn} Given a subgroup $F\subseteq \Zz_{n}\times \Zz_{n}$  we shall say that a 2-cocycle
$\psi\in Z^2(\Zz_{n}\times \Zz_{n},\ku^{\times})$ is \emph{compatible} with $F$ if

\begin{equation}\eqlabel{compatib-psi} q^i\,  \frac{\psi((g,1), f)} {\psi(f, (g,1))} = q^j\,
\frac{\psi((1,g^{-1}),f)} {\psi(f, (1,g^{-1}))}
\end{equation}
for any $f= (g^i, g^j)\in F$. We shall say that
 a 2-cocycle
$\psi\in Z^2(F,\ku^{\times})$ is \emph{compatible} with $F$ if the corestriction 
(see \cite[p. 81]{Br}) 
 of 
$\psi$ in $ Z^2(G,\ku^{\times})$ satisfies
\equref{compatib-psi}.
\end{defn}

\begin{rem} Equation \equref{compatib-psi} is obtained by replacing
the values of $\chi_1, \chi_2$ given in  \equref{charact-psi}, using $\psi^{-1}$, in
equation \equref{coi-stab-F}.

\end{rem}

Let us introduce five families of left $H$-comodule algebras.
\begin{itemize}
 \item Let $F\subseteq \Zz_{n}\times \Zz_{n}$ be a subgroup such that $(g,g)\in F$, $\psi\in
Z^2(F,\ku^{\times})$ a 2-cocycle compatible with $F$, $\xi, \mu\in\ku$ with $\xi\neq 0$. Set $\ele(\xi,\mu, F,\psi)$ for the algebra generated by elements
$\{w,e_f: f\in F\}$ subject to relations
$$w^n=\mu 1, \quad e_f e_{f'}=\psi(f, f')\, e_{ff'},  \quad e_f w=\tau_f\,  w e_f.$$
Here $\tau_f= q^i$ if $f=(g^i, g^j)$.
The left comodule structure $\lambda:\ele(\xi,\mu, F,\psi)\to H\otk \ele(\xi,\mu, F,\psi)$
is defined by
$$\lambda(e_f)=f\ot e_f, \quad \lambda(w)= \xi\, x\ot 1+ y (g,g)\ot e_{(g,g)} + (g,1)\ot w.$$

\item Let $a,b,\xi\in\ku$, $F\subseteq G$ a subgroup, $\psi\in Z^2(F,\ku^{\times})$.
Set $\Kc_{11}(a,b,\xi, F,\psi)$ for the algebra generated by elements
$\{z,u,e_f : f\in F\}$ subject to relations
$$z^n= a\, 1, \quad u^n=b\, 1, \quad zu - uz = \xi\, e_{(g,g^{-1})}, \quad e_f e_{f'}=
\psi(f,f')\, e_{ff'},$$
$$ e_{(g^i,g^j)} z= q^i\, z  e_{(g^i,g^j)}, \quad e_{(g^i,g^j)} u=q^j\, w e_{(g^i,g^j)}
\quad \text{ if } (g^i,g^j)\in F.$$
If $(g,g^{-1})\notin F$ then $\xi=0$. The coaction
$\lambda: \Kc_{11}(a,b,\xi, F,\psi)\to H\otk \Kc_{11}(a,b,\xi, F,\psi)$ is defined by
$$\lambda(e_f)=f\ot e_f, \quad  \lambda(z)= x\ot 1 + (g,1)\ot z,\quad
 \lambda(u)=y\ot 1 + (1,g^{-1})\ot u.$$
\item The algebra $\Kc_{01}(a, F,\psi)$ is the subcomodule algebra of
$\Kc_{11}(a,b,\xi, F,\psi)$  generated by elements
$\{z,e_f : f\in F\}$.

\item The algebra $\Kc_{10}(b, F,\psi)$ is the subcomodule algebra of
$\Kc_{11}(a,b,\xi, F,\psi)$  generated by elements
$\{u,e_f : f\in F\}$.

\item Let $F\subseteq \Zz_{n}\times \Zz_{n}$ be a subgroup, $\psi\in
Z^2(F,\ku^{\times})$ a 2-cocycle then $\ku_\psi F$ is the twisted group algebra.
\end{itemize}

\begin{rem} The first family of comodule algebras is related to the
coideal subalgebra datum of type $\xi$ and the other four families
are related to the
coideal subalgebra datum of type $(1,1)$, $(0,1)$, $(1,0)$ and  $(0,0)$ respectively.
\end{rem}

\begin{lma}\label{r-simp} The  algebras $\ele(\xi,\mu, F,\psi), $ $\Kc_{11}(a,b,\xi, F,\psi),$
$\Kc_{01}(a, F,\psi),$ $\Kc_{10}(b, F,\psi)$ are  right $H$-simple left $H$-comodule algebras with trivial
coinvariants.
\end{lma}
\begin{proof} It is enough to note that
$$\ku_\psi F=\ele(\xi,\mu, F,\psi)_0=\Kc_{11}(a,b,\xi, F,\psi)_0
=\Kc_{01}(a, F,\psi)_0=\Kc_{10}(b, F,\psi)_0$$ and
use  \cite[Prop. 4.4]{M1}.\qed
\end{proof}

Let $\psi\in Z^2(\Zz_{n}\times \Zz_{n},\ku^{\times})$ and $\sigma_\psi:H\otk H\to \ku$ be the associated Hopf
2-cocycle.  Let $\chi_1, \chi_2$ be the characters in $\Zz_{n}\times \Zz_{n}$ defined
in \equref{charact-psi} and let $( W^1, W^2, W^3, F)$ be a coideal subalgebra datum for $H_{(\chi_1,\chi_2)}$.

\begin{lma}\label{tw-coidl} If $W^3=0$   there is an isomorphism of comodule algebras
$$C_{(\chi_1,\chi_2)}(W^1, W^2, 0, F)_{\sigma_\psi^{-1}}\simeq \Kc_{ij}(0,0,0, F,\psi^{-1}) $$ for some
$i,j \in \{0,1\}$.   If $W^3\neq 0$ then
$C_{(\chi_1,\chi_2)}(0, 0, W^3, F)_{\sigma_\psi^{-1}}\simeq \ele(\xi,0 , F,\psi^{-1})$
 for some $\xi\in \ku^{\times}$.\qed
\end{lma}

\subsection{Classification of exact  module categories over  $\Rep(T_q\ot T_{q^{-1}})$}

In this section we give a classification of exact indecomposable $\Rep(T_q)$-bimodule categories.
This is a new result, interesting in itself.

\begin{thm}\thlabel{ex-modcat} Let $\mo$ be an exact indecomposable 
$\Rep(T_q)$-bimodule category
 then $\mo$ is equivalent to one of the following
categories:
\begin{itemize}
\item ${}_{\ku_\psi F}\mo$ for some subgroup $F\subseteq G$ and $\psi\in Z^2(F,\ku^{\times})$;
 \item ${}_{\ele(\xi,\mu, F,\psi)}\mo$ for some subgroup $F\subseteq G$ such that $(g,g)\in F$ and
$\psi\in Z^2(F,\ku^{\times})$ is compatible with $F$,
 $\xi, \mu\in\ku$ with $\xi\neq 0$;
 \item ${}_{\Kc_{11}(a,b,\xi, F,\psi)}\mo$ for some $a,b,\xi\in\ku$, $F\subseteq G$ a subgroup, $\psi\in Z^2(F,\ku^{\times})$;
 \item ${}_{\Kc_{01}(a, F,\psi)}\mo$ for some $a\in\ku$, $F\subseteq G$ a subgroup, $\psi\in Z^2(F,\ku^{\times})$;
\item  ${}_{\Kc_{10}(b, F,\psi)}\mo$ for some $a\in\ku$, $F\subseteq G$ a subgroup, $\psi\in Z^2(F,\ku^{\times})$.
\end{itemize}

\end{thm}
\begin{proof} By Lemma \ref{r-simp} all module categories listed above
are exact indecomposable. Let $\mo$ be an indecomposable exact
 $\Rep(T_q)$-bimodule category, then it is an indecomposable exact 
$\Rep(H)$-module category. By \cite[Thm 3.3]{AM}
there exists a right $H$-simple left comodule algebra with trivial coinvariants $(A,\lambda)$, $\lambda:A\to H\otk A$,
such that $\mo={}_A\mo$ as $\Rep(H)$-modules. Since $H$ is coradically graded then
$\gr A$ is a right $H$-simple left comodule algebra also with trivial coinvariants.
Thus, there exists a subgroup $F\subseteq G$ and $\psi\in  Z^2(F,\ku^{\times})$
such that $\gr A_0=\ku_\psi F$.
\medbreak

Abusing the notation we shall denote by $\psi\in  Z^2(G,\ku^{\times})$ the
2-cocycle such that restricted to $F$  it  equals $\psi$. Since $(\gr A)_{\sigma_\psi^{-1}}$
is a Loewy-graded comodule algebra in $H^{[\sigma_\psi^{-1}]}$, it follows
from \cite[Lemma 5.5]{M4} that $(\gr A)_{\sigma_\psi^{-1}}$ is isomorphic
to a homogeneous coideal subalgebra of $H^{[\sigma_\psi^{-1}]}$.
Let $\chi_1, \chi_2$ be the characters in $G$ defined
in \equref{charact-psi} using $\psi^{-1}$. Then $H^{[\sigma_\psi^{-1}]}= H_{(\chi_1,\chi_2)}$
and by Theorem \ref{class-hcoid} $(\gr A)_{\sigma_\psi^{-1}}= C_{(\chi_1,\chi_2)}(W^1,W^2,W^3,F)$
for some coideal subalgebra datum $(W^1,W^2,W^3,F)$. Then
 $\gr A=C_{(\chi_1,\chi_2)}(W^1,W^2,W^3,F)_{\sigma_\psi}$ and
by Lemma \ref{tw-coidl} there are two options: when $W^3=0$ and $W^3\neq 0$.

\bigbreak

We shall only analyze the case when $W^3\neq 0$, the other case is done similarly.
In this case, by Lemma \ref{tw-coidl},  $\gr A=\ele(\xi,0 , F,\psi)$.

\begin{claim} There exists an element $w\in A$ such that
\begin{equation}\label{co-w} \lambda(w)=  \xi\, x\ot 1+ y (g,g)\ot e_{(g,g)} + (g,1)\ot w.
\end{equation}
\end{claim}

\begin{proof}
Since  $\gr A=\ele(\xi,0 , F,\psi)$ there exists an element $\overline{w}\in A_1/A_0$ such that
$$\overline{ \lambda }(\overline{w})= \xi\, x\ot 1+ y (g,g)\ot e_{(g,g)} + (g,1)\ot \overline{w}.$$
Here $\overline{ \lambda }: \gr A\to H\otk \gr A$ is the coaction induced from $\lambda$,
see for example \cite[Section 4]{M1}.
Let us denote by $w'\in A_1$ a representative element
in the class of $\overline{w}$. By the definition of the induced coaction $\overline{ \lambda }$ there
exists a map $\lambda_1:\gr A\to \ku G\otk  \ku_\psi F$ such that
$$\lambda(w')= \xi\, x\ot 1+ y (g,g)\ot e_{(g,g)} + (g,1)\ot w' + \lambda_1(w').$$
Set $\lambda_1(w')= \sum_{h\in G, f\in F} a_{h,f}\; h\ot e_f$, for some $ a_{h,f}\in \ku$. By the coassociativity of
$\lambda$ we get that $(g,1)\ot \lambda_1(w')+(\id\ot \lambda)  \lambda_1(w')= (\Delta\ot\id)\lambda_1(w')$, whence
$$\lambda_1(w') =\sum_{f\in F, f\neq  (g,1)}
\beta_f\; ((g,1)-f)\ot e_f.$$
Then, if $z= \sum_{f\in F, f\neq  (g,1)}\beta_f\;  e_f\in A_0$ we get that
$ \lambda_1(w') = (g,1)\ot z- \lambda(z).$
The element $w=w'+ z$ satisfies equation \eqref{co-w}.\qed
\end{proof}

It is not difficult to prove that we can choose one $w\in A$ which also
satisfies $e_f w = (f\cdot w) e_f$, where the action $\cdot:F\times W^3\to W^3$
is the restriction of the action of $G$ on $V_1\oplus V_2$. The set $\{f w^i: f\in F, 0\leq i
< n\}$ is a basis for $A$. Since
$$ \lambda(w^n)= 1\ot w^n,$$
there exists $\mu\in \ku$ such that $w^n= \mu 1$. Hence, there is a
projection $\ele(\xi,\mu , F,\psi)\to A$ that must be an isomorphism since both algebras
have the same dimension.\qed
\end{proof}

We shall analize when  module categories listed in \thref{ex-modcat} are equivalent.

\begin{prop}\prlabel{equival-cl} The following statements hold:
\begin{itemize}
 \item[1.] ${}_{\ku_\psi F}\mo\simeq {}_{\ku_{\psi'} F'}\mo$ if and only if  $F=F'$ and $\psi=\psi'$
in $H^2(F,\ku^{\times})$;
 \item[2.]  ${}_{\ele(\xi,\mu, F,\psi)}\mo \simeq {}_{\ele(\xi',\mu', F',\psi')}\mo$
 if and only if $\xi=q^i \xi'$, $\mu=\mu'$ for some $i\in \N$ and $F=F'$ and $\psi=\psi'$
in $H^2(F,\ku^{\times})$;
 \item[3.]  ${}_{\Kc_{11}(a,b,\xi, F,\psi)}\mo\simeq {}_{\Kc_{11}(a',b',\xi', F',\psi')}\mo$
if and only if $(a,b,\xi, F,\psi)= (a',b',\xi', F',\psi')$;
 \item[4.]  for any $(i,j)\in \{(0,1),(1,0)\}$ ${}_{\Kc_{ij}(a, F,\psi)}\mo\simeq {}_{\Kc_{ij}(a', F',\psi')}\mo$
if and only if $(a, F,\psi)=(a', F',\psi')$.
\end{itemize}
\end{prop}

\begin{proof} We shall only prove  (2). The proofs
of the other statements are analogous. It is not difficult to prove that if there is an isomorphism
of left $H$-comodule algebras
  $\ele(\xi,\mu, F,\psi)\simeq \ele(\xi',\mu', F',\psi'),$ then 
$\xi=q^a \xi'$, $\mu=\mu'$, for some $a\in \N$ and $F=F'$, $\psi=\psi'$.
\medbreak

It follows from \cite[Thm. 4.2]{GM} that ${}_{\ele(\xi,\mu, F,\psi)}\mo \simeq {}_{\ele(\xi',\mu', F',\psi')}\mo$
if and only if there exists $f\in G$ such that $\ele(\xi,\mu, F,\psi)\simeq \ele(\xi',\mu', F',\psi')^f$ as
left $H$-comodule algebras, where the latter comodule algebra has the structure as in \equref{g-twisted}.
One readily obtains that if $f\in G$ then $\ele(\xi',\mu', F',\psi')^f=\ele(q^i\xi',\mu', F',\psi')$  for some $i\in \N$.\qed
\end{proof}

One of the consequences of the above is the classification of $T_q$-biGalois objects. The
classification was already obtained by Schauenburg, see \cite{S}.
\begin{cor}  If $A$ is a $T_q$-biGalois object then $A\simeq \ele(\xi,\mu, \diag(G),1)$ as
$T_q$-bicomodule algebras for some $0\neq\xi, \mu\in  \ku$. 
\end{cor}
\begin{proof} Let $A$ be a $T_q$-biGalois object. Then $A$ as a left $T_q\otk T_q^{\cop}$-comodule
algebra has no non-trivial $H$-costable ideal. Indeed, let $I\subseteq A$ be 
an $H$-costable ideal. Thus $I$ is a $T_q$-costable ideal and since $A$ is biGalois, this means that
$I=0$ or $I=A$. Then $A$ must be one of the algebras listed above.  If we observe these algebras 
as $T_q$-bicomodules, it is easily seen that the only family of 
algebras that are $T_q$-biGalois is $\ele(\xi,\mu, \diag(G),1)$ for some $0\neq\xi, \mu\in  \ku$ 
(keep in mind that since $T_q$ is finite-dimensional, every biGalois object is isomorphic to $T_q$ as a bicomodule). 
\qed
\end{proof}

We shall denote $\ele(\xi,\mu)=\ele(\xi,\mu, \diag(G),1)$ for any $0\neq\xi, \mu\in  \ku$.
Two biGalois objects $\ele(\xi,\mu), \ele(\xi',\mu')$
are isomorphic if and only if $\mu=\mu', \xi=q^i\xi'$ for some $i\in \N$. 
Recall the definition of the equivalence relation $\sim$ given in \deref{eq-bigal}. 
 As in the proof of \prref{equival-cl} we have that $\ele(\xi,\mu)\sim  \ele(\xi',\mu')$
if and only if $\mu=\mu', \xi=q^i\xi'$ for some $i\in \N$. It is straightforward to see that 
$T_q\simeq\ele(1,0)$. We then have: 

\begin{cor} 
The subgroup $\inbi(T_q)$ from \coref{bigal seq} is trivial 
and the map $\phi: \biga(T_q) \to  \brp(\Rep(T_q)), 
\phi([A])= [{}_A\mo]$ is a group embedding. 
\end{cor}

\subsection{Invertible $\Rep(T_q)$-bimodule categories}

As a consequence of the above results we give an explicit family of invertible 
exact $\Rep(T_q)$-bimodule categories that form a subgroup inside $ \brp(\Rep(T_q))$.
\medbreak

Define the group $\ku^{\times} \ltimes \ku^{+}$ 
with the underlying set $\ku-\{0\} \times \ku$ and product given by
$$(a,b) \cdot (c,d) = (ac, cb+d), $$
for any $(a,b), (c,d)\in \ku^{\times} \times \ku.$ If $\mathbb{G}_n$ 
denotes the subgroup of $\ku^{\times}$ of n-th roots of unity, then
$\ku^{\times}/\mathbb{G}_n \ltimes \ku^{+}$ is a group with product
$$(\overline{a},b) \cdot (\overline{c},d) = (\overline{ac}, c^n b+d), $$
for any $(\overline{a},b), (\overline{c},d)\in \ku^{\times}/\mathbb{G}_n \times \ku.$
The map $\phi:\ku^{\times}/\mathbb{G}_n \ltimes \ku^{+} \to \ku^{\times} \ltimes \ku^{+}$
given by $\phi(\overline{\xi}, \mu)=(\xi^n,\mu)$ is a group isomorphism.
Schauenburg proved that there 
is a group isomorphism $\biga(T_q) \simeq \ku^{\times} \ltimes \ku^{+}$, \cite[Thm. 2.5]{S}. 
We shall give another proof of this result, mainly for two reasons. Our description of biGalois objects
is different from the one in \cite{S} and we also want to show an explicit subgroup
inside $ \brp(\Rep(T_q))$.

\begin{thm}\thlabel{invt-product} Let $\xi, \mu, \xi', \mu'\in \ku$, $\xi',\xi\neq 0$.
 There
is an isomorphism of $T_q\otk T_{q^{-1}}$-comodule algebras
$$ \ele(\xi',\mu')  \Box_{T_q}  \ele(\xi,\mu)  \simeq \ele(\xi'\xi,\xi^n\mu'+ \mu). $$
\end{thm}
\begin{proof} Recall that the left $T_q\otk T_{q^{-1}}$-comodule structure
on the cotensor product is given by \equref{coact-coprodu-t}. 
Let
$$\gamma: \ele(\xi'\xi,\xi^n\mu'+ \mu)\to \ele(\xi',\mu')  \Box_{T_q}  \ele(\xi,\mu)$$
be the algebra map determined by
$$ \gamma(w)=\xi\; w\ot 1 + e_{(g,g)}\ot w, \quad \gamma(e_f)=e_{f} \ot e_f,$$
for all $f\in \diag(G)$. Note that we are abusing  the notation by denoting with the same
name the generators of the  algebras $\ele(\xi,\mu), \ele(\xi',\mu')$
and $\ele(\xi'\xi,\xi^n\mu'+ \mu)$. To prove that $ \gamma$ is well-defined we have to verify that
$$ \gamma(w^n)= (\xi^n\mu'+ \mu)\, 1, \quad
\gamma(e_{(g,g)} w)=q \, \gamma(we_{(g,g)}).$$
This is done by a straightforward computation. Also, it can be readily
 proven that the image of $\gamma$ is contained in $\ele(\xi',\mu')  \Box_{T_q}  \ele(\xi,\mu)$
and that $\gamma$ is an injective comodule morphism. To prove that $\gamma$ is bijective,
 we shall prove that
$$\dim(\ele(\xi',\mu')  \Box_{T_q}   \ele(\xi,\mu))=\dim( \ele(\xi'\xi,\xi^n\mu'+ \mu)). $$
Since $T_q$ is finite-dimensional, then any Hopf-Galois object is Cleft. This implies that
$\ele(\xi,\mu)\simeq T_q$
as a right  and left $T_q$-comodules. Hence, there
are linear isomorphisms
$$\ele(\xi',\mu')  \Box_{T_q}  \ele(\xi,\mu)\simeq T_q\Box_{T_q}T_q\simeq T_q.$$
Therefore $\dim(\ele(\xi',\mu')  \Box_{T_q}   \ele(\xi,\mu))=\dim( \ele(\xi'\xi,\xi^n\mu'+ \mu)). $\qed
\end{proof}

As a consequence of the above Theorem
there is a group isomorphism $\ku^{\times} \ltimes \ku^{+} \to \biga(T_q)$  given by 
$(\xi,\mu)\mapsto [\ele(\phi^{-1}(\xi,\mu))]$. 

\begin{cor}
There is an injective group
homomorphism $\alpha:\ku^{\times} \ltimes \ku^{+} \to \brp(\Rep(T_q))$
given by 
$$\alpha(\xi,\mu)= 
 [{}_{\ele(\phi^{-1}(\xi,\mu))}\mo ], \quad (\xi,\mu)\in \ku^{\times} \ltimes \ku^{+}.$$\qed
\end{cor}

\subsection*{Acknowledgments} The work of B.F. was supported by a postdoctoral fellowship of
PEDECIBA, Uruguay. The work of M.M. was supported by  Secyt (UNC), Mincyt (C\'ordoba)
and CONICET.
The authors want to thank C. Galindo and R. Heluani for  helpful discussions.
We also thank the referee for his many constructive comments.

\bibliographystyle{amsalpha}

\begin{thebibliography}{AE}




 \bibitem{AM}  {\sc N. Andruskiewitsch } and {\sc M. Mombelli}.
{\it On module categories over finite-dimensional Hopf algebras}.
J. Algebra \textbf{314} (2007), 383--418. 


\bibitem{AS} {\sc N. Andruskiewitsch } and
{\sc H.-J. Schneider}, \emph{Lifting of quantum linear spaces and
pointed Hopf algebras of order $p^3$}, J. Algebra \textbf{209}
(1998), 658--691.

\bibitem{Be} {\sc J. B\'enabou}, \emph{Introduction to bicategories}, 
Lecture notes in mathematics \textbf{47} (1967).

\bibitem{Br} {\sc K. Brown}, \emph{Cohomology of groups}, 
New York: Springer-Verlag, 1982. Print.

\bibitem{BFRS}{\sc T. Barmeier, J. Fuchs, I. Runkel,} and {\sc C. Schweigert}.
\emph{On the Rosenberg-Zelinsky sequence in abelian monoidal categories}. J. 
Reine Angew. Math. \textbf{642} (2010), 1--36.

\bibitem{Bo}{\sc F. Borceux},  \emph{Handbook of Categorical Algebra, Basic Category Theory (Encyclopedia of Mathematics and its Applications)}, 
Volume 1, Cambridge University Press 1994. 


\bibitem{CCMT} {\sc S. Caenepeel, S. Crivei, A. Marcus} and {\sc M. Takeuchi}.
\emph{ Morita equivalences induced by bimodules over Hopf-Galois extensions}. J. Algebra \textbf{314} (2007), 267--30

\bibitem{De}  {\sc P. Deligne} \emph{Cat\`{e}gories tannakiennes}.  The Grothendieck
Festschrift, Vol. II,
Progr. Math., \textbf{87}, Birkh\"auser, Boston, MA, 1990, 111--195.

\bibitem{D} {\sc Y. Doi}. \emph{Unifying Hopf modules}. J. Algebra \textbf{153} (1992), 373--385.

\bibitem{DKR}  {\sc A. Davydov,} {\sc L. Kong} and  {\sc I. Runkel}. \emph{
Invertible defects and isomorphisms of rational CFTs}. ZMP-HH/10-13
Hamburger Beitr\"age zur Mathematik (2010).


\bibitem{ENO}  {\sc P. Etingof}, {\sc D. Nikshych} and {\sc V. Ostrik}. \emph{Fusion categories and
 homotopy theory}. Quantum Topol. \textbf{1}, No. 3, (2010) 209--273.


\bibitem{eo} {\sc P. Etingof} and {\sc V. Ostrik}.
\emph{Finite tensor categories}. Mosc. Math. J. \textbf{4} No. 3 (2004), 627--654.

\bibitem{fsv} {\sc J. Fuchs}, {\sc  C. Schweigert} and
{\sc A. Valentino}. \emph{Bicategories for boundary conditions
and for surface defects in 3-d TFT},
Comm.  Math. Physics
 \textbf{321}, Issue 2, (2013) 543--575.

\bibitem{Ga} {\sc C. Galindo}. \emph{Crossed product tensor categories}.
J. Algebra \textbf{337} (2011) 233--252.



\bibitem{GM} {\sc A. Garc\'ia Iglesias} and {\sc M. Mombelli}.
\emph{Representations of the category of modules over pointed Hopf algebras
 over ${\mathbb S}_3$ and ${\mathbb S}_4$}.  Pacific J. Math. vol \textbf{252} No.  2 (2011)
 343--378.

\bibitem{Gr}  \textsc{J. Greenough}. \emph{Monoidal 2-structure of Bimodule Categories}.
  J. Algebra \textbf{324} (2010) 1818--1859.



\bibitem{KK} {\sc A. Kitaev} and  {\sc L. Kong}. \emph{Models for
 gapped boundaries and domain walls}.
Comm.  Math. Physics
\textbf{313}, Issue 2, (2012) 351--373.






\bibitem{M1} {\sc M. Mombelli}. {\it Module categories over
pointed Hopf algebras}. Math. Z. \textbf{266} No. 2 (2010) 319--344.


\bibitem{M3} {\sc M. Mombelli}. \emph{On the tensor product of bimodule categories
over Hopf algebras}. Abh. Math. Semin. Univ. Hamb. \textbf{82}, (2012)  173--192.

\bibitem{M4} {\sc M. Mombelli}.  \emph{The Brauer-Picard group of the
 representation category of finite supergroup algebras}. Preprint arXiv:1202.6238.


\bibitem{S} {\sc P. Schauenburg}. \emph{Bi-Galois objects over the Taft algebras}.
Israel  J. Math. \textbf{115} (2000), 101--123.


\bibitem{S2} {\sc P. Schauenburg}. \emph{Hopf Bigalois extensions}.
Comm. in Algebra \textbf{24} (1996) 3797--3825.


\bibitem{Sk} {\sc S. Skryabin}. \emph{Projectivity and freeness over comodule algebras}.
 Trans. Am. Math. Soc. 359 \textbf{6} (2007) 2597--2623. 


\bibitem{U} {\sc  K. -H. Ulbrich}. \emph{Galois extensions as functors of comodules}.
Manuscr. Math.
 \textbf{59}, Number 4 (1987), 391--397.


\end{thebibliography}

\end{document}